\documentclass[smallextended,numbook,runningheads]{svjour3}     
\smartqed  
\usepackage{graphicx}
\usepackage{amsmath}
\usepackage{mathptmx}      
\usepackage{amsfonts}
\usepackage{graphicx}
\usepackage{amsmath}
\usepackage{mathptmx}      
\usepackage{amsfonts}
\usepackage{overpic}
\usepackage{multirow}

\usepackage{tikz}

\usepackage{amssymb,amscd}
\usepackage{graphicx}
\usepackage{epstopdf} 
\usepackage{mathptmx}      
\usepackage{color}

\usepackage{mdframed}
\usepackage{systeme,mathtools}
\usepackage{tabstackengine}
\usepackage[]{algorithmic}

\stackMath

\catcode`,\active

\catcode`\,12

\usepackage{graphicx}
\usepackage{amsmath}
\usepackage{mathptmx}      
\usepackage{amsfonts}
\usepackage{booktabs}

\usepackage{graphicx}
\usepackage{enumerate}
\usepackage{color}
\usepackage{overpic}
\usepackage{amssymb,amscd}
\usepackage{epstopdf} 
\usepackage{color}
\usepackage{amsmath,amsfonts}

\usepackage{mdframed}
\usepackage{systeme,mathtools}
\usepackage{tabstackengine}
\usepackage{enumerate}
\usepackage[]{algorithm2e}

\stackMath

\def\un{\hbox{\rm1\kern-.28em\hbox{I}}}
\newcommand*{\QEDB}{\hfill\ensuremath{\square}}

\usepackage{graphicx}
\usepackage{enumerate}
\usepackage{color}
\usepackage{amssymb,amscd}
\usepackage{epstopdf} 
\usepackage{color}
\usepackage{amsmath,amsfonts}
\usepackage{mathtools}

\DeclarePairedDelimiter\floor{\lfloor}{\rfloor}

\def\CC{\hbox{C\kern -.58em {\raise .54ex \hbox{$\scriptscriptstyle |$}}
		\kern-.55em {\raise .53ex \hbox{$\scriptscriptstyle |$}} }}

\def\ie{{\sl \thinspace i.e.},\ }

\def\un{\hbox{\rm1\kern-.28em\hbox{I}}}
\def\un{\hbox{\rm1\kern-.28em\hbox{I}}}

\begin{document}
\title{Computation of optimal linear strong stability preserving methods via 
adaptive spectral transformations of Poisson-Charlier measures
\thanks{ }}
	
\titlerunning{optimal threshold factors and spectral transformations}        
	
\author{Rachid Ait-Haddou}
	
\institute{Rachid Ait-Haddou \at Department of Mathematics and Statistics, King Fahd University of Petroleum $\&$ Minerals, \\ Dhahran 31261, Saudi Arabia.
\email{rachid.aithaddou@kfupm.edu.sa}         
}
	
\date{Received: date / Accepted: date}
	
\maketitle

\begin{abstract}
	Strong stability preserving (SSP) coefficients govern the maximally allowable step-size at which positivity 
	or contractivity preservation of integration methods for initial value problems
	is guaranteed. In this paper, we show that the task of computing linear SSP coefficients of explicit one-step methods 
	is, to a certain extent, equivalent to the problem of characterizing positive quadratures with integer nodes 
	with respect to Poisson-Charlier measures. Using this equivalence, we provide sharp upper and lower bounds for the optimal 
	linear SSP coefficients in terms of the zeros of generalized Laguerre orthogonal polynomials. This in particular provides us 
	with a sharp upper bound for the optimal SSP coefficients of explicit Runge-Kutta methods.
	Also based on this equivalence, we propose a highly efficient and stable algorithm for computing these coefficients, and 
	their associated optimal linear SSP methods, based on adaptive spectral transformations of Poisson-Charlier measures.
	The algorithm possesses the remarkable property that its complexity depends only on the order of the method  
	and thus is independent of the number of stages. 
	Our results are achieved by adapting and extending an ingenious technique by Bernstein in his seminal 
	work on absolutely monotonic functions \cite{bernstein}. Moreover, the techniques introduced in this 
	work can be adapted to solve the integer quadrature problem for any positive discrete multi-parametric 
	measure supported on $\mathbb{N}$ under some mild conditions on the zeros of the associated 
	orthogonal polynomials. 
\end{abstract}  

\section{Introduction}
\label{intro}
Many explicit numerical schemes for solving initial value problems, when applied to  
a linear system of $s\geq1$ ordinary differential equations
\begin{equation}
\label{eq:InitialProblem}
\frac{d}{dt} U(t) = A U(t), \quad t \geq 0, \quad U(0)=u_{0},
\end{equation}
where $A$ is a real $s\times s$ matrix and $u_{0} \in \mathbb{R}^s$,
reduce to a scheme of the type 
\begin{equation}
\label{eq:NumericalScheme}
u_{k}= \phi(hA) u_{k-1}, \quad k=1,2,3,\ldots,
\end{equation}
where $h>0$ is the step-size, $u_{k}$ is an approximation to $U(kh)$, and $\phi$ 
is a polynomial with real coefficients which satisfies 
\begin{equation}
\label{eq:Order}
\phi(x)= \exp(x) + \mathcal{O}(x^{n+1}) 
\quad \textnormal{when} \quad
x \rightarrow 0
\end{equation}
for an integer $n \geq 1$. The greatest integer $n$ for which (\ref{eq:Order}) holds 
is a measure for the local accuracy of the numerical scheme (\ref{eq:NumericalScheme}).   

The matrix $A$ in (\ref{eq:InitialProblem}) and the polynomial $\phi$ in (\ref{eq:NumericalScheme}) 
being given, it is natural to ask for the maximally allowable step-size $h$ at which the numerical 
scheme (\ref{eq:NumericalScheme}) preserves a given property of the exact solution 
to (\ref{eq:InitialProblem}). Prior to giving two prominent examples illustrating such situations, we recall
a few definitions. A $C^{\infty}$ function $f$ is said to be {\it{absolutely monotonic}} over an interval 
$[a,b]$ if, for any $x \in [a,b]$ and for any non-negative integer $k$, $f^{(k)}(x) \geq 0$. 
Denote by $\Pi_{m,n}$, with $m \geq n$, the set of polynomials $\phi$ of degree $m$ ($m \geq 1$) satisfying 
condition (\ref{eq:Order}). The {\it{threshold factor}} or the {\it{linear strong stability preserving (SSP) coefficient}}, 
$R(\phi)$, of a polynomial $\phi$ in $\Pi_{m,n}$ is defined as
\begin{equation*}
R(\phi) = \sup \{ r \; | \; r=0  \; \textnormal{or} \; ( r >0 \;  \textnormal{and} \; \phi \; \textnormal{is absolutely monotonic over }  [-r,0]) \}.
\end{equation*}
Now, let us assume that the matrix $A$ in (\ref{eq:NumericalScheme}) preserves positivity, 
\ie for every initial value $u_{0} \in \mathbb{R}^s$ such that $u_{0} \geq 0$ 
we have $U(t) \geq 0$ for $t \geq 0$. Here and everywhere else the inequalities should be 
interpreted component-wise. It is well known that the matrix $A$ preserves positivity if and only 
if it is a Metzler matrix,\ie the off diagonal elements of $A$ are non-negative \cite{bolley}. 
Moreover, it is shown in \cite{bolley} that, given a Metzler matrix $A$, if the step-size $h$ in 
(\ref{eq:NumericalScheme}) satisfies 
\begin{equation}
\label{positivity}
h \leq \frac{R(\phi)}{\alpha}
\quad \textnormal{with} \quad 
\alpha =  \max_{a_{ii}\leq 0} |a_{ii}|,
\end{equation}
then the numerical scheme (\ref{eq:NumericalScheme}) preserves positivity in the sense that for any 
initial value $u_0 \geq 0$, we have $u_{k} \geq 0$ for any $k\geq 1$. Moreover, the quantity 
$R(\phi)/\alpha$ is the supremum of all the step-sizes that preserve positivity for any Metzler matrix with 
diagonal elements satisfying $a_{ii} \geq -\alpha$.      

Another example where the linear SSP coefficient $R(\phi)$ appears naturally 
is when the matrix $A$ is dissipative with respect to a given norm $|.|$ in 
$\mathbb{R}^s$ \ie for any initial value $u_0 \in \mathbb{R}^s$, the solution to 
(\ref{eq:InitialProblem}) satisfies $|U(t)| \leq |u_0|$ for any $t \geq 0$. It is well known 
that the set of dissipative matrices coincides with the set of matrices satisfying the so-called 
{\it{circle condition}},\ie there exists a positive real number $\beta$ such that 
$|| A + \beta I|| \leq \beta$ where $||.||$ stands for the matrix norm induced by $|.|$ 
and $I$ stands for the identity matrix \cite{spijker}. It is shown in \cite{spijker} 
that if the step-size $h$ in (\ref{eq:NumericalScheme}) satisfies 
\begin{equation}
\label{contractivity}
h \leq \frac{R(\phi)}{\beta}, 
\end{equation}
then the numerical scheme (\ref{eq:NumericalScheme}) preserves contractivity in the sense that, for any 
initial value $u_0$, we have $|u_{k}| \leq |u_{0}|$ for any $k\geq 1$.  Moreover, the quantity $R(\phi)/\beta$ 
is the supremum of all the step-sizes that preserve contractivity for any matrix satisfying the circle 
condition $|| A + \beta I|| \leq \beta$.     

In many practical situations, it is essential to have some flexibility in the choice of the step-size $h$ 
while ensuring the preservation of specific properties of the exact solution. In this respect, 
conditions (\ref{positivity}) and (\ref{contractivity}) suggest to take in (\ref{eq:NumericalScheme}) 
the polynomial $\phi$ that maximizes the value of $R(\phi)$. 
This motivates the introduction of the {\it{optimal threshold factor}} or the {\it{optimal linear SSP coefficient}}, $R_{m,n}$, defined as 
\begin{equation}
\label{ThresholdDefinition}
R_{m,n} =  \sup \{ R(\phi) \;|\; \phi \in \Pi_{m,n} \}.
\end{equation}     
In \cite{kraaPoly} Kraaijevanger showed that $0 < R_{m,n} \leq m-n-1$ and that there exists 
a unique polynomial $\Phi_{m,n}$ in $\Pi_{m,n}$, called the {\it{optimal threshold polynomial}} or the 
{\it{optimal linear SSP polynomial}}, such that 
\begin{equation*}
R(\Phi_{m,n}) = R_{m,n}.
\end{equation*}
Given an $m$-stage Runge-Kutta method with coefficients $(M,b)$ with $M$ an $(m\times m)$ matrix 
and $b$ an $(m\times 1)$ vector and define the $(m+1,m+1)$ matrix  
\begin{equation*}
K = K(M,b) :=  \begin{pmatrix}
M &  0 \\
b^{T} & 0
\end{pmatrix}.
\end{equation*}   
The {\it{strong stability preserving (SSP)} coefficient}, $R(M,b)$, of the Runge-Kutta method is defined by 
{\fontsize{9}{12} 
	\begin{equation*}
	R(M,b) :=  \sup \{ r \; | \; \forall \rho \in [0,r], \; (I + \rho K)^{-1} \; \textnormal{exists}, \; 
	\rho K(I + \rho K)^{-1} \geq 0 \; \; \textnormal{and} \; \; \rho K(I + \rho K)^{-1} e \leq e \},  
	\end{equation*}}
where $e = (1,1,\ldots,1) \in \mathbb{R}^{m+1}$ and vectors and matrix inequalities are understood
component-wise. The SSP coefficient $R(M,b)$ plays the same role in numerical positivity and contractivity 
preservation for non-linear problems as $R_{m,n}$ does for linear problems. In particular, 
for explicit $m$-stages Runge-Kutta methods of order $n$ with $m \geq n$, we have 
\begin{equation}
\label{RKversus}
R(M,b) \leq R_{m,n}.
\end{equation}
Various studies have investigated optimal SSP methods for one-step, multi-stages 
methods \cite{kraaPoly,KraaRatio,ketcheson} and one stage, multi-step methods \cite{lenferink,lenferink2,ketcheson2}.
Systems of the type (\ref{eq:InitialProblem}) appear in semi-discretization, 
discontinuous Galerkin semi-discretization or spectral semi-discretization of partial differential equations
\cite{chen,gottlieb2,lu,gottlieb3,gottlieb4}. Optimally contractive schemes for solving these systems are 
important in so far as they prevent the growth of propagated errors.     

In the present paper we study the size of the optimal linear SSP coefficients and their associated methods for
multi-stages, one-step methods. The case of multi-step methods will be the subject of another paper. 
To present our main results, we first recall 
that the generalized Laguerre polynomials 
$L^{(\gamma)}_{n}$ are orthogonal on the interval $[0,\infty)$ with respect to the weight 
$x^{\gamma} e^{-x}$, that is,
\begin{equation}
\label{LaguerreIntegral}
\int_{0}^{\infty} L_{n}^{(\gamma)}(x) L_{m}^{(\gamma)}(x)  
x^{\gamma} e^{-x} dx = 0, \quad \textnormal{if} \quad n \not= m.  
\end{equation}
The integral in (\ref{LaguerreIntegral}) converges only if $\gamma>-1$. The zeros of generalized Laguerre polynomials 
are positive real numbers and {\bf{throughout this work we will denote by $\ell_{n}^{(\gamma)}$ the smallest zero of the 
generalized Laguerre polynomial $L_{n}^{(\gamma)}$}}. Let us also recall that Poisson-Charlier polynomials $C_{n}(.,R)$ are orthogonal 
polynomials with respect to the discrete Poisson-Charlier measure $\mu_{R}$ given by 
\begin{equation}
\label{CharlierPoissonMeasure}
\mu_{R} = e^{-R} \sum_{j=0}^{\infty} \frac{R^j}{j!} \delta_{j},
\end{equation}
where $\delta_{j}$ is the Dirac measure. In Section 2 and Section 3 we re-visit the work of Kraaijevanger \cite{kraaPoly}
with a new formalism that fits best our narrative. In Section 4 we establish a connection between the task of 
computing the optimal threshold factors and the notion of positive quadratures with integer nodes with respect 
to Poisson-Charlier measures. This connection leads to our first main result.   
\begin{theorem}
\label{MainTheorem1}
For any positive integers $m$ and $p$ such that $m \geq 2p-1$ 
\begin{equation}
\label{UpperBound}
R_{m,2p-1} \leq \ell_{p}^{(m-p)},  
\end{equation}
with equality if and only if the zeros of the Poisson-Charlier polynomial 
$C_{p}(.,\ell_{p}^{(m-p)})$ are integers.
\end{theorem}
Table \ref{tab:table1} shows some of the exact values of the optimal threshold factors $R_{m,n}$ 
(computed using the algorithm described in Section 8) and the
upper bound obtained in Theorem \ref{MainTheorem1}. The quality of the upper bound (\ref{UpperBound}) is 
rather remarkable and surprising. Moreover, according to Theorem \ref{MainTheorem1} and (\ref{RKversus}), 
we have the following. 
\begin{corollary}
The SSP coefficient of an explicit $m$-stage Runge-Kutta method of order $2p-1$ ($m \geq 2p-1$) satisfies 
\begin{equation}
\label{RKbound}
R(M,b) \leq  \ell_{p}^{(m-p)}.
\end{equation}
\end{corollary}
The use of the upper bound (\ref{RKbound}) has to take into account the order barrier for explicit SSP Runge-Kutta methods \cite{mainKra,ketcheson3}.

From now on, we shall call the optimal SSP coefficients and associated polynomials simply by optimal threshold factors and polynomials, respectively.   

An attempt at finding an equally satisfying lower bound for the optimal threshold
factors as in Theorem \ref{MainTheorem1} is the object of Section 5. Using duality concepts, 
we prove the following result.
\begin{theorem}
	\label{MainTheorem2}
	For any positive integers $m$ and $p$ such that $m \geq 2p-1$ 
	\begin{equation}
	\label{LowerBound}
	R_{m,2p-1} \geq \ell_{p}^{(m-2p+1)}.
	\end{equation}
\end{theorem}
Although the lower bound given in (\ref{LowerBound}) is sharp (we have equality in (\ref{LowerBound}) when $p=1$), 
it is not as impressive as the upper bound obtained in (\ref{UpperBound}) (see Table \ref{tab:table1}). 
Nevertheless,  we give strong evidences of possible improvements of the lower bound in (\ref{LowerBound}). More precisely, we show that 
\begin{equation*}
R_{m,2p-1} \geq \ell_{p}^{(m-p-\tau_{p})},
\end{equation*}
where 
\begin{equation*}
\tau_{p} := \sup_{R>0} \sup_{P\in \mathcal{C}_{p}}
\left( \frac{\int t P(t) d\mu_{R}(t)}{\int P(t) d\mu_{R}(t)} - \lambda_{p,p}(R) \right),
\end{equation*}
where $\lambda_{p,p}(R)$ is the largest zero of the Poisson-Charlier polynomial $C_p(.,R)$ and 
$\mathcal{C}_{p}$ is the set of non-zero real polynomials of degree at most $2p-2$ that are non-negative on
$\mathbb{Z}$. The lower bound (\ref{LowerBound}) is obtained by showing that
$\tau_{p} \leq p-1$ using Cauchy residue theorem. 
\begin{table}[h!]
	\label{tab:table1}
	\centering
	\begin{tabular}{|c|c|c|} \hline
		$\ell_{p}^{(m-2p+1)}$ & $R_{m,2p-1}$ & $\ell_{p}^{(m-p)}$ \\ \hline
		11.0108  & $R_{20,5}$=12.5512    & 12.6118\\
		19.1884  & $R_{30,5}$=20.8355    & 20.8659\\
		9.7026   & $R_{22,7}$=11.8435    & 11.9237\\
		23.6589  & $R_{40,7}$=26.0713    & 26.0927\\
		44.4670  & $R_{65,7}$=47.0065    & 47.0267\\
		3.6304   & $R_{16,9}$=5.9337     & 6.0762\\
		41.7638  & $R_{67,9}$ =45.0148   & 45.0533\\
		10.5254  & $R_{30,11}$=13.8617   & 13.9257\\
		17.4568  & $R_{40,11}$=21.0411   & 21.0911\\ \hline
	\end{tabular}
	\smallskip
	\caption{Upper and lower bounds for the optimal threshold factor $R_{m,2p-1}$.}
\end{table}

In \cite{kraaPoly} it is shown that the $R_{m,n}$-table of the optimal threshold 
factors enjoys a remarkable property of stabilization along the diagonal,\ie for any non-negative $d$,
there exists an integer $p=p(d)$ such that 
\begin{equation}
\label{diagonalProp}
R_{p+d+k,p+k} = R_{p+d,p} \quad \textnormal{for all} \quad k \geq 1.
\end{equation}
In Section 6, we adapt and extend  an ingenious technique by Bernstein in his seminal work \cite{bernstein}
to identify a structural property of the optimal threshold polynomials. This leads to 
the following surprising property of the optimal threshold factors that in some sense complements the diagonal
stability property (\ref{diagonalProp}).
\begin{theorem}
	\label{MainTheorem3}	
	The optimal threshold factors $R_{m,n}$ ($m\geq n \geq 1$) are algebraic numbers such that  
	\begin{equation}
	\label{smalldiagonal1}
	R_{m+1,2p} = R_{m,2p-1},
	\end{equation}
	for any positive integers $m$ and $p$ such that $m \geq 2p-1$. Moreover, the associated optimal 
	threshold polynomials satisfy the relation
	\begin{equation}
	\label{OptimalRelation}
	\Phi_{m+1,2p}(x) = 1 + \int_{0}^{x} \Phi_{m,2p-1}(\xi) d\xi.  
	\end{equation}
\end{theorem}
Note that (\ref{smalldiagonal1}) and (\ref{OptimalRelation}) assert that it is enough to compute $R_{m,n}$ and $\Phi_{m,n}$ 
for odd integers $n$ to obtain the whole $R_{m,n}$-table of optimal threshold factors and their associated 
optimal threshold polynomials. This is an essential property that will prove extremely useful in this work. 
Namely that we shall find it more natural to study the optimal threshold factors $R_{m,n}$ with $n$ 
an odd integer than if $n$ is an even integer.  

The structural property of the optimal threshold polynomial asserts the following fundamental result proven in Section 6.
\begin{theorem}
	\label{fundamental-corollary}
	For any positive integers $m$ and $p$ such that $m \geq 2p-1$, 
	the optimal threshold polynomial $\Phi_{m,2p-1}$ has the form 
	\begin{equation}
	\label{introth}
	\Phi_{m,2p-1}(x) = \sum_{k=1}^{2p-1} \alpha_k \left(1 + \frac{x}{R_{m,2p-1}}\right)^{m_k},
	\end{equation}
	where $\alpha_k, k=1,2,\ldots,2p-1$ are non-negative real numbers with $\alpha_{2p-1} >0$
	and the integers $0\leq m_1<m_2<\ldots<m_{2p-1}\leq m$ satisfy
	\begin{equation}
	\label{introm}
	m_{2k} = m_{2k-1} +1, \quad k=1,2,\ldots,p-1 \quad \textnormal{and} \quad m_{2p-1} = m.
	\end{equation}
\end{theorem}
Theorem \ref{fundamental-corollary} sates that in the representation (\ref{introth}) of the polynomial $\Phi_{m,2p-1}(x)$, 
the integers $(m_1,m_2,\ldots,m_{2p-2})$ come in pairs of consecutive integers and that $m_{2p-1} = m$. 
For example, using the algorithm described in Section 8, one can show that $R_{100,5} \simeq 83.002$ is the unique positive 
zero of the polynomial
\begin{equation*}
R^5 - 394R^4 + 62544R^3 - 5001012R^2 + 201456936R - 3271262400,
\end{equation*}
and the optimal threshold polynomial $\Phi_{100,5}$ is given by
\begin{equation*}
\Phi_{100,5} (x) =  \sum_{k=1}^{5}\alpha_k \left(1 + \frac{x}{R_{100,5}}\right)^{m_k},
\end{equation*}
where 
\begin{equation*}
(m_1,m_2,m_3,m_4,m_5) = (68,69,83,84,100),
\end{equation*}
and 
\begin{equation*}
 (\alpha_1,\alpha_2,\ldots,\alpha_5) = ( 0.1188,0.0765,0.2095,0.4539,0.1413).
\end{equation*}
The structural form (\ref{introth}) is further analyzed in Section 7 to reveal a set of rigid
rules on the allowable values of the integers $m_i, i=1,\ldots,2p-1,$ in the representation (\ref{introth}). Loosely stated,
we shall show a tight connection between the location of the integers $m_i, i=1,2,\ldots 2p-1$, the zeros of 
Poisson-Charlier polynomials and the zeros of the orthogonal polynomials associated with Christoffel transforms 
of Poisson-Charlier measures. This is achieved through a comprehensive study of specific spectral transformations 
of Charlier-Poisson measures. Our analysis leads to a highly efficient and stable algorithm for computing the optimal threshold factors and 
their associated optimal polynomials via adaptive spectral transformations of Poisson-Charlier measures.
{\bf{The algorithm has the particularity that its complexity depends only on the order of approximation 
and not of the degree of the polynomials}} and will be described in Section 8. To put into perspective 
the importance of the complexity of our algorithm, we compared the execution time of our algorithm with a recent algorithm 
in \cite{ketcheson} (within the same computational environment). The computation of $R_{2000,7}$ took 1 hour 30 minutes with
the algorithm in \cite{ketcheson}, while it took 1.2 seconds with ours. By increasing the degree of polynomials, we found that 
the algorithm in \cite{ketcheson} took about 4 hours 20 minutes for the computation of $R_{4000,7}$,
while it took 0.18 seconds with ours. We conclude with future work in Section 9.     
\section{Touchard Polynomials and optimal threshold factors}
\label{sec:2}
Denote by $(x)_{h}$ the Pochhammer symbol,\ie $(x)_{h}=x(x-1)\ldots(x-h+1)$ 
for $h \geq 1$ and $(x)_{0} =1$. The Stirling numbers, $s(n,k)$, of the first kind
and the Stirling numbers, $S(n,k)$, of the second kind are defined as the coefficients
in the expansions
\begin{equation}
\label{stirlings}
(x)_{n} = \sum_{k=0}^{n} s(n,k) x^{k};
\quad x^n = \sum_{k=0}^{n} S(n,k) (x)_k, \quad
n \geq 0; \; \textnormal{for any} \; 
x \in \mathbb{R}.
\end{equation}
The univariate Touchard\footnote{These polynomials are called
Stirling polynomials in \cite{kraaPoly}.} polynomials $B_{n}$ are defined by 
$B_{n}(x)=\sum_{k=0}^{n}S(n,k) x^{k}$ and satisfy the recurrence
\begin{equation}
\label{Bell_recurrence}
B_{0}(x) = 1, \quad B_{n+1}(x) = x \left(B_{n}(x) + B_{n}^{(1)}(x) \right),
\end{equation}
where the notation $F^{(k)}$ refers to the $k$-{th} derivative of the function $F$. 
The following useful relations hold:
\begin{equation}
\label{stirling-bell}
\sum_{k=0}^{n} s(n,k) B_{k}(x) = x^{n}; \quad
\sum_{k=0}^{n} S(n,k) x^{k} = B_{n}(x).
\end{equation}
For a fixed real number $R$, we define the polynomials $\mathcal{H}_n(.;R)$ by 
\begin{equation}
\label{hpolynomial}
\mathcal{H}_n(x;R) = \sum_{k=0}^{n} (-1)^{n-k} \binom{n}{k} B_{n-k}(R) x^{k}.
\end{equation}
The following result is implicit in \cite{kraaPoly}, however for the sake of completeness and also 
due to the difference between our presentation and the one in \cite{kraaPoly}, we provide a proof 
for the result. 
\begin{proposition}
	Let $m$ and $n$ be two positive integers such that $m \geq n$ and $R$ 
	be a positive real number. The polynomial  $\mathcal{H}_n(.;R)$ admits a 
	representation of the form 
	\begin{equation}
	\label{hsylvester}
	\mathcal{H}_n(x;R) = \sum_{i=1}^{s} \alpha_i (x-m_i)^{n}, 
	\quad s \geq 1,
	\end{equation}
	where $\alpha_1,\alpha_2,\ldots,\alpha_s$ are non-negative numbers and 
	where the integers $m_{1},m_{2},\ldots,m_{s}$ are such that  
	$0 \leq m_1 < m_2 < \ldots<m_s \leq m$ if and only if the polynomial 
	\begin{equation}
	\label{phi-polynomial}
	\Phi(x) = \sum_{i=1}^{s} \alpha_i \left(1 + \frac{x}{R} \right)^{m_i}
	\end{equation}
	is of degree at most $m$, is absolutely monotonic over the interval $[-R,0]$ and 
	it satisfies $\Phi(x) - e^{x} = \mathcal{O}(x^{n+1})$ as $x\rightarrow 0$.
\end{proposition}
\begin{proof}
	Let us assume that the polynomial $\mathcal{H}_n(.;R)$ admits a representation
	of the form (\ref{hsylvester}). We have 
	\begin{equation*}
	\Phi^{(\ell)} (x) = \sum_{i=1}^{s} \frac{\alpha_{i} (m_{i})_{\ell}}{R^{\ell}}
	\left( 1 + \frac{x}{R} \right)^{m_{i}-\ell}.
	\end{equation*} 
	Thus, $\Phi^{(\ell)} (-R) = 0$ if $\ell \notin \{m_{1},m_{2},\ldots,m_{s} \}$ and 
	$\Phi^{(m_{i})}(-R) = \alpha_{i} (m_{i})_{m_{i}}/R^{m_{i}} \geq 0$ for $i=1,2,\ldots,s$.
	Therefore, the polynomial $\Phi$ is absolutely monotonic at $-R$ and hence is 
	absolutely monotonic over the interval $[-R,0]$ (see Lemma 4.3 in \cite{Kra2}). 
	Moreover, from (\ref{hpolynomial}) and (\ref{hsylvester})
	we have  $\sum_{i=1}^{s} \alpha_{i} m_{i}^{\ell} = B_{\ell}(R)$ 
	for $\ell =0,1,\ldots,n$. Therefore, using (\ref{stirlings}) and (\ref{stirling-bell}), 
	we obtain for $\ell=0,1,\ldots,n$,   
	\begin{equation}
	\label{zero-derivatives}
	\Phi^{(\ell)} (0) = \frac{\sum_{i=1}^{s} \alpha_{i} (m_{i})_{\ell}}{R^{\ell}} =  
	\frac{ \sum_{j=1}^{\ell} s(\ell,j) \sum_{i=1}^{s} \alpha_{i} m_{i}^{j}} {R^{\ell}}= 
	\frac{\sum_{j=1}^{\ell} s(\ell,j) B_{j}(R)}{R^{\ell}} =1.  
	\end{equation}    
	Therefore, we have $\Phi(x) - e^{x} = O(x^{n+1})$ as $x \rightarrow 0$. 
	Conversely, given real numbers $\alpha_1,\ldots,\alpha_s$ and given integers 
	$0\leq m_1 <m_2<\ldots< m_s \leq m$, assume that the corresponding polynomial $\Phi$ in 
	(\ref{phi-polynomial}) is absolutely monotonic over $[-R,0]$ and that it satisfies  
	$\Phi(x) - e^{x} = \mathcal{O}(x^{n+1})$ as $x\rightarrow 0$. Then necessarily the coefficients 
	$\alpha_{i},i=1,2,\ldots,s$ are non-negative. Denote by $\mathcal{A}_{n}(.;R)$ 
	the polynomial $\mathcal{A}_n(x;R) = \sum_{i=1}^{s} \alpha_i (x-m_i)^{n}$.
	Using (\ref{stirlings}) the coefficient $a_{j}$ attached to the monomial $x^{n-j}$ of 
	the polynomial $\mathcal{A}_n(.;R)$ is given by 
	\begin{equation*}
	a_{j} = (-1)^{j} \binom{n}{j} \sum_{i=1}^{s} \alpha_{i} m_{i}^{j}=
	(-1)^{j} \binom{n}{j} \sum_{\ell=0}^{j} S(j,\ell) \sum_{i=1}^{s} 
	\alpha_{i} (m_{i})_{\ell}.
	\end{equation*}
	According to (\ref{zero-derivatives}), we have 
	$\sum_{i=1}^{s} \alpha_{i} (m_{i})_{\ell} = R^{\ell}$. 
	Thus, form (\ref{stirling-bell}) we obtain
	\begin{equation*}
	a_{j} =  (-1)^{j} \binom{n}{j} \sum_{\ell=0}^{j} S(j,\ell) R^{\ell} = 
	(-1)^{j} \binom{n}{j} B_{j}(R). 
	\end{equation*}
	Therefore, the coefficient $a_{j}$ coincide with the coefficient of 
	$x^{n-j}$ of the polynomial $\mathcal{H}_n(.;R)$
	given in (\ref{hpolynomial}). Hence, the polynomials $\mathcal{A}_n(.;R)$ 
	and $\mathcal{H}_n(.;R)$ coincide. 
\qed
\end{proof}
From the previous proposition, the optimal threshold factor 
$R_{m,n}$ defined in (\ref{ThresholdDefinition}) can also be characterized 
as follows.
\begin{corollary}
	\label{HR}
	Let $m$ and $n$ be positive integers such that $m \geq n$. 
	The optimal threshold factor $R_{m,n}$ is the maximum of positive real 
	numbers $R$ for which the polynomial $\mathcal{H}_n(.;R)$ admits a representation 
	of the form 
	\begin{equation}
	\label{TemporarySylvester}
	\mathcal{H}_n(x;R) = \sum_{i=1}^{s} \alpha_i (x-m_i)^{n},  \quad s \geq 1,
	\end{equation}
	with integers $\; 0 \leq m_{1} < m_{2}< \ldots <m_{s}\leq m$ and non-negative 
	real numbers $\alpha_{1},\alpha_{2},\ldots,\alpha_{s}$.
\end{corollary}
In \cite{kraaPoly} Kraaijevanger showed that the optimal threshold polynomial 
$\Phi_{m,n}$ satisfies the property that at least $(m-n+1)$ numbers of the sequence 
$\{\Phi_{m,n}^{(k)}(-R_{m,n})\}_{k=0}^{m}$ vanish. In terms of the polynomial 
$\mathcal{H}_n(.;R_{m,n})$ this claim can be re-stated as saying that for $R = R_{m,n}$, the number of summands in the right-hand side of (\ref{TemporarySylvester}) is at most $n$. More precisely, 
we have the following theorem.   

\begin{theorem}
	\label{Kmaintheorem}
	For any positive integers $m$ and $n$ such that $m \geq n$, there exist integers 
	$0\leq m_{1} < m_{2}< \ldots <m_{n}\leq m$ and non-negative real numbers 
	$\alpha_{1},\alpha_{2},\ldots,\alpha_{n}$ such that 
	\begin{equation*}
	\mathcal{H}_n(x;R_{m,n}) = \sum_{i=1}^{n} \alpha_i (x-m_i)^{n}.
	\end{equation*}
\end{theorem}
\begin{proof}
	Let us assume that $\mathcal{H}_n(x;R_{m,n}) = 
	\sum_{i=1}^{s} \alpha_i (x-m_i)^{n}$ with $s >n$ and  
	$\alpha_{i} > 0, i=1,2,\ldots,s$. Write 
	\begin{equation}
	\label{hsystem}
	\sum_{i=1}^{n} \alpha_i (x-m_i)^{n} =
	\mathcal{H}_n(x;R_{m,n}) - \sum_{i=n+1}^{s} \alpha_i (x-m_i)^{n}.
	\end{equation}
	Equation (\ref{hsystem}) can be viewed as a linear system in 
	$(\alpha_{1},\alpha_{2},\ldots,\alpha_{n})$, i.e; 
	\begin{equation*}
	\sum_{i=1}^{n} \alpha_i m_i^{j} = B_{j}(R_{m,n}) - 
	\sum_{i=n+1}^{s} \alpha_i m_i^{j}, \quad j=0,1,\ldots,n,
	\end{equation*}  
	that has a positive solution,\ie $\alpha_{i} > 0$ for $i=1,2,\ldots,n$. 
	Therefore, there exists an $\epsilon > 0$ such that the linear system in 
	$(\beta_{1},\beta_{2},\ldots,\beta_{n})$ 
	\begin{equation*}
	\sum_{i=1}^{n} \beta_i (x-m_i)^{n} = 
	\mathcal{H}_n(x;R_{m,n}+\epsilon) - \sum_{i=n+1}^{s} \alpha_i (x-m_i)^{n} 
	\end{equation*}
	also has a positive solution. Thus, we obtain
	\begin{equation*} 
	\mathcal{H}_n(x;R_{m,n}+\epsilon) =  \sum_{i=1}^{n} \beta_i (x-m_i)^{n} 
	+ \sum_{i=n+1}^{s} \alpha_i (x-m_i)^{n}.
	\end{equation*}
	This contradicts the definition of $R_{m,n}$ as given in Corollary \ref{HR}.
\qed
\end{proof}  
\section{Polar forms and Kraaijevanger's algorithm}
\label{sec:3}
Polar forms (or blossoms) for polynomials \cite{ramshaw} are compelling tools in various 
mathematical areas \cite{aithaddou1,aithaddou2,aithaddou3,aithaddou4}. They will prove helpful, 
even essential, at several places in this work. In the present section, after a brief reminder 
of their definition, we will use them to give a simple description of the algorithm proposed by Kraaijevanger \cite{kraaPoly} for computing the optimal threshold factors.

\medskip 

{\bf{Notation: }}Throughout the article, for any real number $x$ and any 
non-negative integer $k$, $x^{[k]}$ will stand for $x$ repeated $k$ times.    

\begin{definition}
	Given a real polynomial $P$ of degree at most $n$, there exists a unique 
	symmetric multi-affine function $p(u_1,u_2,\ldots,u_n)$ such that 
	$p(x^{[n]}) = P(x)$ for any $x \in \mathbb{R}$. The function $p$ 
	is called the {\it{blossom}} or the {\it{polar form}} of the polynomial $P$.
\end{definition}
The polar form of a polynomial $P$ expressed in the monomial basis as 
$P(x) = \sum_{k=0}^{n} a_{k} x^{k}$ is given by 
\begin{equation*}
p(u_1,u_2,\ldots,u_n) = \sum_{k=0}^{n} a_{k} \sigma_{k}(u_1,u_2,\ldots,u_n),
\end{equation*}
where $\sigma_{k}$ refers to the normalized $k$-th elementary symmetric polynomial,\ie
\begin{equation*}
\sigma_{k}(u_1,u_2,\ldots,u_n) = \binom{n}{k}^{-1} \sum_{1 \leq j_{1} <\ldots<j_{k} \leq n} 
u_{j_{1}} u_{j_{2}} \ldots u_{j_{k}}.
\end{equation*}
Of special interest within this work are polynomials of the form 
\begin{equation*}
P(x) = \sum_{k=1}^{s} \alpha_k (x-a_k)^{n}.
\end{equation*}
Their polar forms are simply given by 
\begin{equation}
\label{PolarProduct}
p(u_1,u_2,\ldots,u_n) = \sum_{k=1}^{s} \alpha_k \prod_{i=1}^{n} (u_{i}-a_k).
\end{equation}

We shall need the following proposition.  
\begin{proposition}
	\label{Prop:Polar}
	Let $P$ be a real polynomial of degree at most $n$ and $p$ its polar form. 
	Given any pairwise distinct real numbers $\xi_1, \ldots, \xi_k$, we have  
	\begin{equation}
	\label{polar_condition}
	p(\xi_1,\xi_2,\ldots,\xi_k,x^{[n-k]}) = 0 
	\quad \textnormal{for any} \quad x \in \mathbb{R}, 
	\end{equation}
	if and only if the polynomial $P$ can be written in the form
	$\displaystyle{P(x) = \sum_{j=1}^k \alpha_j (x-\xi_j)^{n}}$.
\end{proposition}
\begin{proof}
	The function $\widetilde P(x):= p(\xi_1, \ldots, \xi_k, x^{[n-k]})$ is a polynomial of degree at most $(n-k)$.
	Select any pairwise distinct $\xi_{k+1}, \ldots, \xi_n$ in $\mathbb{R}\setminus\{\xi_1,\ldots,\xi_k\}$. 
	Let us expand $P$ as $P(x)=A+\sum_{i=1}^n\alpha_i(x-\xi_i)^n$. Then, from (\ref{polar_condition}) we obtain
	\begin{equation*}
	\widetilde P(x)=A+\sum_{i=k+1}^n \beta_i(x-\xi_i)^{n-k}, \quad x\in\mathbb{R},
	\end{equation*}
	with $\beta_i:=\alpha_i\prod_{j=1}^k(\xi_j-\xi_i), i=k+1, \ldots, n$. 
	Accordingly, the polynomial $\widetilde P$ is identically zero if and only all coefficients 
	$A, \beta_{k+1}, \ldots, \beta_n$ are zero, that is, 
	if and only if $A$ and $\alpha_{k+1},\ldots,\alpha_n$ are zero. The claim is proved. 
\qed      
\end{proof} 

Now, we are in a position to describe the algorithm of Kraaijenvanger for computing 
the optimal threshold factor $R_{m,n}$ and the associated polynomial $\Phi_{m,n}$.
If we write the polynomial $\mathcal{H}_n(.;R)$ defined in (\ref{hpolynomial}) 
in the form 
\begin{equation}
\label{Tempoh}
\mathcal{H}_n(x;R) = \sum_{i=1}^{n} \alpha_i (x-m_i)^{n},
\end{equation}
then, by denoting $h_{n}(u_{1},u_{2},\ldots,u_{n};R)$ the value at 
$(u_1,u_2,\ldots,u_n)$ of the polar form of the polynomial $\mathcal{H}_n(.;R)$ 
and applying (\ref{PolarProduct}), we obtain  
\begin{equation}
\label{zerocondition}
h_{n}(m_{1},m_{2},\ldots,m_{n};R)=0.
\end{equation}
Moreover, evaluating the polar form of both sides of (\ref{Tempoh}) at 
$(m_{1},m_{2},\ldots,m_{k-1},m+1,m_{k+1},\ldots,m_{n})$ yields
\begin{equation}
\label{positivitycondition}
\alpha_{k} = \frac{h_{n}(m_{1},m_{2},\ldots,m_{k-1},m+1,m_{k+1},\ldots,m_{n};R)}
{(m+1-m_{k}) \prod_{i=1, i\not=k}^{n} (m_{i} - m_{k})}, \quad k=1,2,\ldots,n.
\end{equation}
Based on (\ref{zerocondition}) and (\ref{positivitycondition}), an algorithm 
for computing $R_{m,n}$ goes as follows:

\vskip 0.3 cm

\noindent
$\bullet$ {\bfseries{Step 1}}:  Generate all integer sequences $M = (m_{1},m_{2},\ldots,m_{n})$ 
such that $0\leq m_{1}<m_{2}<\ldots<m_{n} \leq m$. For each such sequence, 
find the positive numbers $R$ satisfying (\ref{zerocondition}) (if any). 
Note that for each such integer sequence $M$, Equation (\ref{zerocondition}) is a polynomial 
equation of degree $n$ in $R$. 

\vskip 0.3 cm

\noindent
$\bullet$ {\bfseries{Step 2}}: For each of the real numbers $R$ found in Step 1,  
check the non-negativity of the coefficients $\alpha_{k}$ using equations (\ref{positivitycondition}).
Retain the numbers $R$ and the associated sequences $M$ for which all the coefficients $\alpha_{k}$ 
are non-negative. 

\vskip 0.3 cm

\noindent
$\bullet$ {\bfseries{Step 3}}: $R_{m,n}$ is the maximum of all the values $R$ that 
survived elimination from Step 2. The optimal threshold polynomial is then given by 
\begin{equation*}
\Phi_{m,n}(x) = \sum_{k=1}^{n} \alpha_k \left(1 + \frac{x}{R_{m,n}} \right)^{m_i}
\end{equation*}
where $(m_1,m_2,\ldots,m_n)$ the integer sequence associated with $R_{m,n}$ and  
$\alpha_{k}$ are the coefficients that were already computed using (\ref{positivitycondition}).    

\vskip 0.3 cm

Evidently, the computational cost of the above algorithm grows exponentially in $m$ and $n$ 
and could only be used to compute the optimal threshold factors for very small values of 
$m$ and $n$. As will be clear later, the above algorithm can be substantially improved by 
our results of Section 6 where we identify a structural property of the optimal threshold 
polynomials that considerably reduces the number of integer sequences $M = (m_{1},m_{2},\ldots,m_{n})$ 
to be considered in Step 1 of the algorithm. We will have further comments on this aspect of 
the algorithm but we should stress that in Section 8, we propose a highly efficient algorithm
for the computation of $R_{m,n}$ whose computational cost is independent of the integer $m$. 
We would like to mention that a method of computing $R_{m,n}$ and the associated optimal 
threshold polynomial based on linear programming is presented in \cite{ketcheson}. 
However, the algorithm in question suffers from stability problems for large value
of the integer $m$.

\section{Poisson-Charlier orthogonal polynomials and sharp upper bounds for the optimal threshold factors}
In this section, we give a connection between Poisson-Charlier orthogonal 
polynomials and the polynomials $\mathcal{H}_n(.;R)$ defined in (\ref{hpolynomial}). 
This will enable us to give a sharp upper bound for the optimal threshold factors in terms
of the smallest zero of generalized Laguerre polynomials. 

The monic Poisson-Charlier polynomials $C_{n}(.,R)$ are orthogonal with respect to the 
discrete Poisson-Charlier measure (\ref{CharlierPoissonMeasure}). Thus, they satisfy the 
orthogonality relations \cite[pp. 170]{chihara}  
\begin{equation*}
\int C_{n}(t,R) C_{m}(t,R) d\mu_{R}(t) =  \sum_{j=0}^{\infty}  
C_{n}(j,R) C_{m}(j,R)e^{-R}\frac{R^{j}}{j!} = n! R^{n}\delta_{nm}.   
\end{equation*}
The Poisson-Charlier polynomials satisfy the three-term recurrence relation
\begin{equation}
\label{recurrence}
t C_{n}(t,R) = C_{n+1}(t,R) + (R+n)C_{n}(t,R) + nR C_{n-1}(t,R).
\end{equation}
with $C_0(t,R) = 1$ and  $C_{1}(t,R) = t-R$.

It is well known that the moments of Poisson-Charlier measures are Touchard 
polynomials.  However, as we were not able to find a reference for a proof of this fact,
we include a simple one for the readers convenience.   
\begin{proposition}\label{proposition_monomial}
	For any non-negative integer $n$ and positive number $R$, the following relations hold
	\begin{equation}
	\label{monomial_expression}
	t^{n} = \sum_{j=0}^{n} \frac{B_{n}^{(j)}(R)}{j!} C_{j}(t,R)
	\quad \textnormal{and} \quad
	\int  t^{n} d\mu_{R}(t) = B_{n}(R).
	\end{equation}
\end{proposition}
\begin{proof}
	The proof of the left identity in (\ref{monomial_expression}) proceeds by induction 
	on the integer $n$. The identity being trivial for $n=0$, assume that it holds for 
	any $k \leq n$. The three-term recurrence relation (\ref{recurrence}) yields
	\begin{equation*}
	t^{n+1}= \sum_{j=0}^{n}  \frac{B_{n}^{(j)}(R)}{j!} t C_{j}(t;R) 
	=\sum_{j=0}^{n+1} a_{j} C_{j}(t,R), 
	\end{equation*}
	with 
	\begin{equation*}
	a_{j} = \frac{B_{n}^{(j-1)}(R)}{(j-1)!}+(R+j) \frac{B_{n}^{(j)}(R)}{j!} 
	+(j+1) R \frac{B_{n}^{(j+1)}(R)}{(j+1)!}
	\quad \textnormal{for}  \quad j \geq 0, \; \textnormal{with} \; B^{(-1)}_{n} \equiv 0. 
	\end{equation*}  
	Form the recurrence equation (\ref{Bell_recurrence}) of Touchard polynomials, we have
	\begin{equation*}
	B_{n+1}^{(j)}(R) = j B_{n}^{(j-1)}(R) + (j+R) B_{n}^{(j)}(R) + R B_{n}^{(j+1)}(R).
	\end{equation*}
	Thus, for $j=0,1,\ldots,n$, $a_{j} = B_{n+1}^{(j)}(R)/j!$. The right identity in  
	(\ref{monomial_expression}) is a direct consequence of the left identity and of the 
	orthogonality of the Poisson-Charlier polynomials. This concludes the proof.
\qed
\end{proof}

\begin{corollary}
	\label{hintegral}	
	For any positive integer $p$, the polynomials $\mathcal{H}_{2p-1}(.;R)$ 
	defined in (\ref{hpolynomial}) can be expressed as 
	\begin{equation}
	\label{hGauss}
	\mathcal{H}_{2p-1}(x;R) = \int  (x-t)^{2p-1} d\mu_{R}(t) = \sum_{k=1}^{p} \omega_k (x - \lambda_{k,p})^{2p-1},
	\end{equation}
	where $\lambda_{1,p}<\lambda_{2,p}<\ldots<\lambda_{p,p}$ are the zeros of the Poisson-Charlier 
	polynomial $C_{p}(.,R)$ and $(\omega_1,\omega_2,\ldots,\omega_p)$ are the positive 
	weights of the $p$-point Gaussian quadrature with respect to the measure $\mu_{R}$. 
\end{corollary}
\begin{proof}
	The first identity in (\ref{hGauss}) is valid when we replace $2p-1$ by any integer $n$. Indeed, 
	according to Proposition \ref{proposition_monomial}, we have
	\begin{equation*}
	\begin{split}
	\int (x-t)^{n} d\mu_{R}(t) & =\sum_{k=0}^n (-1)^{n-k}\binom{n}{k}x^k \int t^{n-k} d\mu_{R}(t)  \\
	& =\sum_{k=0}^n (-1)^{n-k}\binom{n}{k} x^k B_{n-k}(R)=\mathcal{H}_{n}(x;R).
	\end{split}
	\end{equation*}
	The second identity in (\ref{hGauss}) is nothing but the $p$-point Gaussian quadrature 
	with respect to the measure $\mu_{R}$ applied to the polynomial $P(x) = \int (x-t)^{2p-1} d\mu_{R}(t)$. 
\qed
\end{proof}

\begin{remark}
Writing $\mathcal{H}_{n}(x;R) = \sum_{i=1}^{s} \alpha_i (x - m_i)^{n}$
is equivalent to saying that for any polynomial $P$ of degree at most $n$ we have 
\begin{equation*}
\int P(t) d\mu_{R}(t) =  \sum_{i=1}^{s} \alpha_i  P(m_{i}).
\end{equation*}
Therefore, according to Corollary \ref{HR}, the optimal threshold factor $R_{m,n}$ is  
{\it{the maximum of the real numbers $R$ for which the corresponding Poisson-Charlie measure
$\mu_{R}$ admits a positive quadrature with non-negative integer nodes that are smaller 
or equal to $m$}}.
\end{remark}

We shall often use the following theorem which we state in it full generality. Let $\mu$ be a finite positive measure 
over the real line with finite moments of all order.  Denote by $\pi_1,\pi_2,\ldots,\pi_n,...$,
the orthogonal polynomials associated with $\mu$. Denote by $\mathbb{P}_n$ the space of polynomials 
of degree at most $n$.

\begin{theorem}
Given a positive quadrature rule with respect to the measure $\mu$ which is exact in $\mathbb{P}_{2p-1}$, 
\begin{equation}
\label{generalquadrature}
\int P(t) d\mu = \sum_{k=0}^{s} \beta_k P(\rho_k) 
\quad \textnormal{for any} \quad P \in \mathbb{P}_{2n-1},
\end{equation} 
with  $s \geq p$, $0\leq\rho_1<\rho_2<\ldots<\rho_s$ and $\beta_j >0, j=1,2,\ldots,s$. Then 
$\rho_s \geq t_{p}$ and $\rho_1 \leq t_1$ where $t_1 < t_2 < \ldots < t_p$ are the zeros of
the orthogonal polynomial $\pi_n$. Moreover, $\rho_s = t_{p}$ if and only if 
the quadrature (\ref{generalquadrature}) coincides with the Gaussian quadrature. 
\end{theorem}     
\begin{proof}
Consider the polynomial 
\begin{equation}
Q(t) := \pi_n(t) \frac{\pi_n(t)}{t - t_p} = (t-t_p) \prod_{i=1}^{p-1}(t - t_i)^2. 
\end{equation}
Applying (\ref{generalquadrature}) and invoking orthogonality, we obtain 
\begin{equation}
\label{majoration1}
\int Q(t) d\mu =  \sum_{k=1}^{s} \beta_k (t_1-\rho_k)^2 
\ldots (t_{p-1}-\rho_k)^2 (t_{p}-\rho_k)=0.
\end{equation}
Therefore, there exists an integer $k \in \{1,2,\ldots,s\}$ such that 
$\rho_k \geq t_{p}$. In particular, we have $\rho_s \geq t_{p}$. 	
If $\rho_s = t_{p}$ then from (\ref{majoration1}) we remark that 
for $k=1,2,\ldots,s-1$, $(\rho_k-t_{1}) \ldots (\rho_k-t_{p-1}) =0$. 
In other words, $\rho_{1},\rho_{2},\ldots,\rho_{s-2}$ and $\rho_{s-1}$ are zeros of the polynomial 
$\psi(t) = (t-t_{1})(t-t_{2}) \ldots (t-t_{p-1})$. Thus, we necessarily 
have $s-1 = p-1$ and $\rho_{k} = t_{k}$ for $k=1,2,\ldots,p-1$. Therefore, the quadrature 
(\ref{generalquadrature}) coincides with the Gaussian quadrature. To prove that  
$t_{1} \geq \rho_1$ we proceed as follows: Define the polynomial 
\begin{equation*}
S(t) = \pi_n(t) \frac{\pi_n(t)}{t - t_1} = (t-t_1) \prod_{i=2}^{p}(t - t_i)^2. 
\end{equation*}
Applying (\ref{generalquadrature}) and invoking orthogonality, we obtain
\begin{equation*}
\int S(t) d\mu = \sum_{k=1}^{s} \beta_k (t_{2}-\rho_k)^2 \ldots (t_{p}-\rho_k)^2  (t_{1}-\rho_k)=0.
\end{equation*}
Therefore, there exists an integer $k \in \{1,2,\ldots,s\}$ such that 
$\rho_k \leq t_{1}$. In particular, we have $\rho_1 \leq t_{1}$.
\qed
\end{proof}

Applying the previous Theorem to the Poisson-Charlier measures and using the integral representation (\ref{hGauss}) of the polynomial 
$\mathcal{H}_{2p-1}(.;R)$ we readily obtain the following. 

\begin{corollary}
\label{last-term}
Let us assume that the polynomial $\mathcal{H}_{2p-1}(.;R)$ is written as
\begin{equation}
\label{g-gaussian}
\mathcal{H}_{2p-1}(x;R) = \sum_{k=1}^{s} \beta_k (x - \rho_k )^{2p-1},
\end{equation}
with $s \geq p$, $0\leq\rho_1<\rho_2<\ldots<\rho_s$ and $\beta_j >0, j=1,2,\ldots,s$. 
Let $\lambda_{p,p}$ the largest zero of the Poisson-Charlier polynomial $C_{p}(.,R)$. 
Then $\rho_s \geq \lambda_{p,p}$ with equality if and only if $s=p$ and the representation 
(\ref{g-gaussian}) coincides with the one in (\ref{hGauss}). 
Moreover, $\rho_1 \leq \lambda_{1,p}$ where $\lambda_{1,p}$ is the smallest zero of
$C_{p}(.,R)$.    
\end{corollary}

We shall need the following result.  

\begin{corollary}
\label{main-corollary1}
Let $R_{max}$ be the unique real number such that the largest zero of the 
Poisson-Charlier polynomial $C_{p}(.,R_{max})$ is equal to $m$. 
Then $R_{m,2p-1} \leq R_{max}$ with equality if and only if all the zeros 
of $C_{p}(.,R_{max})$ are integers.
\end{corollary}
\begin{proof}
	For a real number $R$, let us denote by $\lambda_{p,p}(R)$ the largest zero of 
	the Poisson-Charlier polynomial $C_{p}(.,R)$. From the definition of $R_{m,2p-1}$, there exist 
	non-negative integers $0 \leq m_{1} < m_{2}<\ldots < m_{s} \leq m$ such that 
	\begin{equation}
	\label{representation2}
	\mathcal{H}_{2p-1}(x;R_{m,2p-1}) = \sum_{k=1}^{s} \beta_k (x-m_k )^{2p-1}
	\end{equation}
	with $\beta_{k} > 0$ for $k=1,2,\ldots,s$ ($s \leq 2p-1$). 
	If $R_{m,2p-1} > R_{max}$ and since the zeros of Poisson-Charlier polynomials are 
	strictly increasing functions of the parameter $R$ (see \cite{area}) , 
	we deduce that $\lambda_{p,p}(R_{m,2p-1}) > \lambda_{p,p}(R_{max}) = m$.
	However, from Proposition \ref{last-term}, we have $m_{s} \geq \lambda_{p,p}(R_{m,2p-1})$. Thus, 
	we obtain $m_{s} >m$ contradicting our initial assumption on $m_s$. 
	Moreover, from Proposition \ref{last-term}, $\lambda_{p,p}(R_{m,2p-1}) = \lambda_{p,p}(R_{max})$ 
	or equivalently $R_{m,2p-1} = R_{max}$ if and only if 
	the two representations (\ref{representation2}) and (\ref{hGauss}) coincide, or equivalently
	the zeros of the polynomial $C_{p}(.,R_{max})$ are integers.
\qed
\end{proof}  	   

\noindent
{\bfseries{Example of Applications: }}
Corollary \ref{main-corollary1} shows that, if for a positive real number $R$ the zeros 
$\lambda_{1,p} < \lambda_{2,p} <\ldots <\lambda_{p,p}$ of the polynomial $C_{p}(.,R)$ are 
integers then
\begin{equation}
\label{general-form}
R_{\lambda_{p,p},2p-1} = R
\quad \textnormal{and} \quad 
\Phi_{\lambda_{p,p},2p-1}(x) = \sum_{k=1}^{p} \omega_{k} 
\left( 1+\frac{x}{R_{\lambda_{p,p},2p-1}}\right)^{\lambda_{k,p}},
\end{equation}
where $\omega_{k}, k=1,2,\ldots,p$ are the weights of the $p$-point Gaussian quadrature 
with respect to the measure $\mu_{R_{\lambda_{p,p},2p-1}}$. As an application, we now prove the 
following theorem which was derived in \cite{kraaPoly} using a technically involved method.

\begin{theorem}
	\label{IntegerCharlier}
	For any integer $m \geq 1$, we have $R_{m,1} = m$ and 
	\begin{equation}
	\label{Rm1}
	\Phi_{m,1}(x) = \left( 1 + \frac{x}{m} \right)^{m}.
	\end{equation}
	For any square integer $m \geq 3$, we have 
	$R_{m,3} = m - \sqrt{m}$ and 
	\begin{equation*}
	\Phi_{m,3}(x) = \frac{\sqrt{m}}{2\sqrt{m}-1} \left( 1 + \frac{x}{m-\sqrt{m}} \right)^{m- 2\sqrt{m}+1}
	+ \frac{\sqrt{m}-1}{2\sqrt{m}-1} \left( 1 + \frac{x}{m-\sqrt{m}} \right)^{m}.
	\end{equation*}
\end{theorem} 
\begin{proof}
	For any non-negative integer $m$, we have $C_{1}(t,m) = t - m$. Thus, according to  (\ref{general-form}) with 
	$\lambda_{1,1} = m$, we have $R_{m,1} =m$. The expression of $\Phi_{m,1}$ in (\ref{Rm1}) is a direct consequence
	(\ref{general-form}).
	The degree $2$ Poisson-Charlier polynomial is given by $C_{2}(t,R) = t^{2} -(2R+1)t + R^{2}$. Thus, 
	for $R = m - \sqrt{m}$ with  $m\geq 3$ is a square integer, we have 
	\begin{equation*}
	C_{2}(t,R) = t^{2} - \left( 2(m - \sqrt{m})+1 \right) t + (m - \sqrt{m})^2 = \left(t-(m- 2\sqrt{m}+1)\right)(t - m).
	\end{equation*}      
	Thus, for these specific values of the parameter $R$, the zeros of $C_{2}(.,R)$ are integers with $m$ as the largest 
	one. Therefore, according to (\ref{general-form}), we have $R_{m,3} = m - \sqrt{m}$. 
	The expression of $\Phi_{m,3}$ is a direct consequence of (\ref{general-form}) once the weights 
	of the $2$-point Gaussian quadrature with respect to $\mu_{m-\sqrt{{m}}}$
	are computed explicitly.  
\qed   
\end{proof}

\begin{remark}
	It is an interesting problem to find all the real numbers $R$ and positive integers $p$ for which 
	all the zeros of the Poisson-Charlier polynomial $C_{p}(.,R)$ are integers. For these cases, 
	the optimal threshold factors and their associated optimal polynomials are easily computed 
	via Gaussian quadratures. It may be possible that the only cases for which all the zeros of 
	$C_{p}(.,R)$ are integers are actually the cases already cited in Theorem \ref{IntegerCharlier}.
\end{remark}
\smallskip

We are now in a position to prove Theorem \ref{MainTheorem1} (see Introduction).

\smallskip 

\noindent
{\bf{Proof of Theorem \ref{MainTheorem1}: }} As is well known, the Poisson-Charlier 
polynomials are linked to the generalized Laguerre polynomials \cite{szego} via the relation 
\begin{equation}
\label{charlier-laguerre}
C_{p}(x,R) = p ! L_{p}^{(x-p)}(R).
\end{equation}
From Corollary \ref{main-corollary1}, $R_{m,2p-1} \leq R_{max}$ where $R_{max}$ is the unique real number 
for which the largest zero of $C_{p}(.,R_{max})$ is equal to $m$. In other words, and taking into account 
that the zeros of $C_{p}(.,R)$ are increasing functions on the parameter $R$, $R_{max}$ is the smallest real
number satisfying $C_{p}(m,R_{max})=0$. Thus, due to (\ref{charlier-laguerre}), $R_{max}$ is the smallest real number 
such that $L_{p}^{(m-p)}(R)=0$, \ie $R_{max} = \ell_{p}^{(m-p)}$. This shows inequality (\ref{UpperBound}). 
The claim about equality in (\ref{UpperBound}) stated in Theorem \ref{MainTheorem1} 
is a direct consequence of Corollary \ref{main-corollary1}.      
\section{Lower bounds for the optimal threshold factors}

The good quality of the sharp upper bound (\ref{UpperBound}) to the optimal threshold 
factor $R_{m,2p-1}$ (see Table \ref{tab:table1}) suggests the possibility of 
finding an equally satisfying lower bound for $R_{m,2p-1}$. This section is an attempt to 
finding such lower bounds. The results of this section are based on the following 
characterization of the optimal threshold factors.

\begin{theorem}
\label{FarkasLemma}
Let $m \geq n$ be two positive integers. The SSP coefficient $R_{m,n}$ is the maximum of the real numbers $R$ 
with the property that if a polynomial $f$ of degree at 
most $n$ is such that $f(j) \geq 0$ for $j=0,1,\ldots,m$, then
\begin{equation*}
\int f(t) d\mu_{R}(t) \geq 0.
\end{equation*}
\end{theorem}
\begin{proof}
Let $R$ be a real number such that $R \leq R_{m,n}$. According to Corollary \ref{HR}, 
the polynomial $\mathcal{H}_{n}(.;R)$ has a representation of the form 
\begin{equation*}
\mathcal{H}_{n}(x;R) = \sum_{k=1}^{s} \beta_k (x - m_k )^{n},
\end{equation*} 
with $0\leq m_{1} < m_{2} < \ldots <m_{s} \leq m$ and $\beta_{k} \geq 0$ for $k=0,1,\ldots,s$.
In other words, for any polynomial $f$ of degree at most $n$, we have (see Remark 4.1)
\begin{equation}
\label{positive-integral}
\int f(t) d\mu_{R}(t) = \sum_{k=1}^{s} \beta_{k} f(m_{k}).
\end{equation}
In particular, if the polynomial $f$ is such that $f(j)\geq 0$ for $j=0,1,\ldots,m,$  
then by (\ref{positive-integral}),  $\int f(t) d\mu_{R}(t) \geq 0$.
Conversely, let us assume that $R$ is such that, for any polynomial $f$ of degree at most $n$ such that $f(j) \geq 0$ for 
$j=0,1,\ldots,m$, we have $\int f(t) d\mu_{R}(t) \geq 0$. Let us write $f$ as 
$f(t) = \sum_{k=0}^{n} \gamma_{k} t^{k}$. Set
$\gamma := (\gamma_{0},\gamma_{1},\ldots,\gamma_{n})^{T}$ and by $A$ the $(n+1,m+1)$ matrix 
$A ={(a_{ij})}$ with $a_{ij}= j^i$ for $ i=0,1,\ldots,n$ and $j=0,1,\ldots,m$. Clearly, we have 
$\gamma^{T} A = \left( f(0),f(1),\ldots,f(m)\right)$.  Moreover, if we denote by 
$b = \left( B_{0}(R),B_{1}(R),\ldots,B_{n}(R)\right)^{T}$, then using the fact that the moments of the 
Poisson-Charlier measures are Touchard polynomials, we obtain 
$\gamma^{T} b = \int f(t) d\mu_{R}(t)$. Thus our initial hypothesis can be restated as: 
for any $\gamma \in \mathbb{R}^{n+1}$ such that $\gamma^{T} A \geq 0$ we have $\gamma^{T} b \geq 0$. 
Therefore, by Farkas lemma, this is equivalent to the existence of a vector 
$\alpha = \left(\alpha_{0},\alpha_{1},\ldots,\alpha_{m}\right)^{T} \geq 0$ 
such that $A\alpha = b$, which in turn is equivalent to the representation 
of the polynomial $\mathcal{H}_{n}(.;R)$ as  
\begin{equation*}
\mathcal{H}_{n}(x;R) = \sum_{k=0}^{m} \alpha_k (x - k )^{n}.
\end{equation*}  
Therefore, $R \leq R_{m,n}$. This concludes the proof.
\qed
\end{proof}

We shall need the following definition.

\begin{definition}
A non-zero real polynomial $P$ is said to be admissible if it is non-negative 
on the set of the integers $\mathbb{Z}$, i.e.; 
\begin{equation*}
P(j) \geq 0 \quad \textnormal{for any} \quad j \in \mathbb{Z}.
\end{equation*} 
\end{definition}

Let $p$ be a positive integer. Denote by $\mathcal{C}_{p}$ the set of all admissible 
polynomials of degree at most $2p-2$. Define the following quantity 
\begin{equation}
\label{TauDefinition}
\tau_{p} := \sup_{R>0} \sup_{P\in \mathcal{C}_{p}}
\left( \frac{\int t P(t) d\mu_{R}(t)}{\int P(t) d\mu_{R}(t)} - \lambda_{p,p}(R) \right),
\end{equation}
where $ \lambda_{p,p}(R)$ is the largest zero of the Poisson-Charlier polynomial $C_p(.,R)$.
In the rest of this section, we shall prove that the quantity $\tau_{p}$ is bounded above, 
and it is even smaller that $p-1$. The relevancy of the quantity $\tau_p$ in establishing a lower bound 
for the linear SSP coefficient is the following theorem.     

\begin{theorem}
\label{lowerbound1}
For any integers $m \geq 2p-1$, we have
\begin{equation*}
R_{m,2p-1} \geq \ell_{p}^{(m-p-\tau_{p})}
\end{equation*}
provided that $m-p-\tau_{p} >-1$.
\end{theorem}
\begin{proof}
Let $\bar{R}$ be the unique real number such that $\lambda_{p,p}(\bar{R}) = m - \tau_{p}$. 
Similar arguments as in the proof of Theorem \ref{MainTheorem1} show that 
$\bar{R} = \ell_{p}^{(m-p-\tau_{p})}$. Let $f$ be a polynomial of degree $2p-1$ such that
$f(j) \geq 0$ for $j=0,1,\ldots,m$. 
Then $f$ can be written as $f = f_{1} f_{2}$ where $f_{1}$ 
is an admissible polynomial of degree $2s$ ($0 \leq s \leq p$) with zeros in the interval $[0,m]$ 
(in case $s=0$, take $f_1 \equiv 1$) and $f_{2}$ is a polynomial of degree $2(p-s)-1$ with 
no zeros in the interval $[0,m]$. Thus necessarily
\begin{equation}
\label{f2positivity}
f_{2}(x) > 0 
\quad \textnormal{for any} \quad 
x \in [0,m].
\end{equation} 
Denote by $\tilde{\mu}_{\bar{R}}$ the positive measure 
\begin{equation*}
\tilde{\mu}_{\bar{R}}= e^{-\bar{R}} \sum_{j=0}^{\infty} \frac{f_{1}(j) \bar{R}^{j}}{j!} \delta_{j},
\end{equation*}
and by $(\tilde{\pi}_{1},\tilde{\pi}_{2},\ldots, \tilde{\pi}_{n}, \ldots)$ 
the sequence of orthogonal polynomials with respect to the measure $\tilde{\mu}_{\bar{R}}$. 
By Gauss quadrature with respect to the measure $\tilde{\mu}_{\bar{R}}$, we have
\begin{equation}
\label{integralpositivity}
\int f(t) d\mu_{\bar{R}}(t) = \int f_{2}(t) f_{1}(t) d\mu_{\bar{R}}(t) = \int f_2(t) d\tilde{\mu}_{\bar{R}}(t) 
=\sum_{j=1}^{p-s} \beta_{j} f_{2}(\tilde{\lambda}_{j}), 
\end{equation}
where $\beta_{j} > 0, j=1,\ldots,p-s$ and where
$\tilde{\lambda}_{1} < \tilde{\lambda}_{2} < \ldots <\tilde{\lambda}_{p-s}$
are the zeros of the orthogonal polynomial $\tilde{\pi}_{p-s}$. 
If we show that $\tilde{\lambda}_{p-s} \leq m$ then, on account of (\ref{f2positivity}),
the integral in (\ref{integralpositivity}) will be non-negative and by Theorem \ref{FarkasLemma}, 
we will have $R_{m,2p-1} \geq \bar{R} = \ell_{p}^{(m-p-\tau_{p})}$. Let us thus assume the opposite,\ie 
$\tilde{\lambda}_{p-s} > m$. Define the admissible polynomial $Q$ of degree $2p-2$ by 
\begin{equation*}
Q(t) = \prod_{j=1}^{p-s-1} (t - \tilde{\lambda}_{j})^2 f_{1}(t).
\end{equation*}
By orthogonality with respect to the measure  $\tilde{\mu}_{\bar{R}}$, we can state that
\begin{equation*}
\int (t-\tilde{\lambda}_{p-s}) Q(t) d\mu_{\bar{R}} = \int \tilde{\pi}_{p-s}(t) \prod_{j=1}^{p-s-1}
(t - \tilde{\lambda}_{j}) d\tilde{\mu}_{\bar{R}}=0.
\end{equation*}
Thus, 
\begin{equation*}
\frac{\int t Q(t) \mu_{\bar{R}}(t)}{\int Q(t) \mu_{\bar{R}}(t)} - \lambda_{p,p}(\bar{R})
= (\tilde{\lambda}_{p-s} -m)+ \tau_{p} > \tau_{p}.
\end{equation*}
This contradicts the definition of $\tau_{p}$. Thus we conclude that $\tilde{\lambda}_{p-s} \leq m$ 
	and the proof is complete.
\qed 
\end{proof}

To give an upper bound for the quantity $\tau_{p}$ defined in (\ref{TauDefinition}), 
we need several preliminary results. Let $P$ be an admissible polynomial and
denote by $(\Pi_{k})_{k \geq 0}$ a sequence of orthogonal polynomials with respect 
to the measure
\begin{equation*}
d\bar{\mu}_{R} = e^{-R} \sum_{j=0}^{\infty} \frac{P(j) R^{j}}{j!} \delta_{j}.
\end{equation*}
Moreover, denote by $(\Pi^{+}_{k})_{k\geq 0}$ the sequence of orthogonal polynomials with 
respect to the measure
\begin{equation*}
d\bar{\mu}^{{+}}_{R} = e^{-R} \sum_{j=0}^{\infty} \frac{P^{+}(j) R^{j}}{j!} \delta_{j},
\quad \textnormal{where} \quad P^{+}(t) =  P(t-1).
\end{equation*}
We need the following comparison result. 
\begin{proposition}
\label{Value1}
Denote by $\bar{\lambda}_{k}$ (resp. $\bar{\lambda}_{k}^{+}$), $k=1,2,\ldots,p, $ 
the zeros of the orthogonal polynomial ${\Pi}_{p}$ (resp. ${\Pi}^{+}_{p}$), listed in increasing order. Then 
\begin{equation*}
\bar{\lambda}_{p}^{+} \leq \bar{\lambda}_{p}+1.
\end{equation*}
\end{proposition} 
\begin{proof}
Applying Gauss quadrature, we obtain 
\begin{equation}
\label{Forgotten}
G_{2p-1}(x) := \int (x-t)^{2p-1} d\bar{\mu}_{R}(t)
= \sum_{i=1}^{p} \beta_{i} (x-\bar{\lambda}_{i})^{2p-1}, 
\end{equation}
with $\beta_{i} > 0, i=1,\ldots,p$. Consider the polynomial $G^{+}_{2p}$ 
of degree $2p$ defined by
\begin{equation}
\label{G1Equation}
G^{+}_{2p}(x) := \int (x-t)^{2p} d\bar{\mu}^{+}_{R}(t)
\end{equation}
and let $g_{2p}$ be its polar form. We have 
\begin{equation*}
\begin{split}
g^{+}_{2p}(0,x^{[2p-1]}) & = - \int t (x-t)^{2p-1} d\bar{\mu}^{+}_{R}(t) \\
& = - \sum_{j=0}^{\infty} \frac{j (x-j)^{2p-1} P(j-1) R^j}{j!} =  -R {G}_{2p-1}(x-1).
\end{split} 
\end{equation*}
Thus, using (\ref{Forgotten}), we obtain
$g^{+}_{2p}(0,\bar{\lambda}_{1}+1,\ldots,\bar{\lambda}_{p}+1, x^{[p-1]}) \equiv 0.$
Therefore, by Proposition \ref{Prop:Polar}, there exist real numbers 
$\alpha_{0}, \alpha_{1}, \ldots \alpha_{p}$ such that  
\begin{equation}
\label{G2Equation}
G^{+}_{2p}(x) = \alpha_{0} x^{2p} + \sum_{k=1}^{p} \alpha_{k} (x - (\bar{\lambda}_{k}+1))^{2p}.
\end{equation}
Using the fact that 
\begin{equation*}
g^{+}_{2p}(\bar{\lambda}_{1}+1^{[2]},\ldots,\bar{\lambda}_{p}+1^{[2]}) = 
\alpha_{0}\prod_{k=1}^{p} (\bar{\lambda}_{k}+1)^{2} = 
\int \prod_{k=1}^{p} (\bar{\lambda}_{k} +1 - t)^{2}d\bar{\mu}^{+}_{R}(t) \geq 0,
\end{equation*} 
we conclude that $\alpha_{0} \geq 0$. Similarly, with 
	$\Lambda = (0^{[2]},\ldots,\bar{\lambda}_{i-1}+1^{[2]},\bar{\lambda}_{i+1}+1^{[2]},
	\ldots,\bar{\lambda}_{p}+1^{[2]})$, we have 
	\begin{equation*}
	g^{+}_{2p}(\Lambda)  = \alpha_{i}(\bar{\lambda}_i+1)^2 \prod_{k=1,k\not=i}^{p} 
	(\bar{\lambda}_{k}-\bar{\lambda}_i)^{2} = \int t^2 \prod_{k=1,k\not=i}^{p} 
	(\bar{\lambda}_{k} +1 - t)^{2}d\bar{\mu}^{+}_{R}(t) \geq 0.
	\end{equation*} 
	Thus we conclude that $\alpha_i \geq 0$ for $i=1,2,\ldots,p$. 
	Differentiating (\ref{G1Equation}) and (\ref{G2Equation}) with respect to the variable $x$,
	and applying Gauss quadrature, we obtain   
	\begin{equation}
	\label{transfer}
	\alpha_{0} x^{2p-1} + \sum_{k=1}^{p} \alpha_{k} (x - (\bar{\lambda}_{k}+1))^{2p-1}
	=\sum_{i=1}^{p} \gamma_{i} (x-\bar{\lambda}^{+}_{i})^{2p-1},
	\end{equation}
	with $\gamma_i > 0$ for $i=1,2,\ldots,p$. Evaluating the polar form of 
	both side of (\ref{transfer}) at 
	$({\bar{\lambda}^{+ [2]}_{1}},\ldots,{{\bar{\lambda}^{+ [2]}}_{p-1}},\bar{\lambda}^{+}_{p})$
	shows that $\bar{\lambda}^{+}_{p} \leq \bar{\lambda}_{p} +1$.
\qed
\end{proof} 
The only instance of the previous proposition that we shall need is the $p=1$ case.
The unique zero $\bar{\lambda}_1$ of the polynomial $\Pi_{1}$ is given by the 
condition
\begin{equation*}
\int (t-\bar{\lambda}_{1}) d\bar{\mu}_{R} = \int (t-\bar{\lambda}_{1}) P(t) d\mu_{R} = 0. 
\end{equation*} 
Thus 
\begin{equation*}
\bar{\lambda}_1 =  \frac{\int t P(t) d\mu_{R}(t)}{\int P(t) d\mu_{R}(t)}.
\end{equation*}
In this specific situation, Proposition \ref{Value1} states that, for any admissible 
polynomial $P$, we have 
\begin{equation*}
\frac{\int t P(t) d\mu_{R}(t)}{\int P(t) d\mu_{R}(t)} \leq  
\frac{\int t P(t+1) d\mu_{R}(t)}{\int P(t+1) d\mu_{R}(t)} + 1.
\end{equation*}
Iterating this inequality leads to the following result.

\begin{corollary}
	\label{iterated1}
	Let $R$ be a positive real number. Then for any admissible polynomial $P$ 
	and for any non-negative integer $j$, we have  
	\begin{equation*}
	\frac{\int t P(t) d\mu_{R}(t)}{\int P(t) d\mu_{R}(t)} \leq 
	\frac{\int t P(t+j) d\mu_{R}(t)}{\int P(t+j) d\mu_{R}(t)} + j.
	\end{equation*} 
\end{corollary}

We shall need the following proposition whose proof was kindly provided to us 
by Fedja Nazarov \cite{aithaddou}.
\begin{proposition}
	\label{Conjecture}
	Let $P$ be an admissible polynomial of degree at most $2n$. Then the polynomial 
	\begin{equation*}
	Q(t) = \sum_{k=0}^{n} \binom{n}{k}^2  P(t+k)
	\end{equation*}
	is non-negative on the whole real line,\ie $Q(t) \geq 0$ for all $t \in \mathbb{R}$.
\end{proposition}
\begin{proof}
	Let $t_{0}$ be an arbitrary real number in $\mathbb{R}/\mathbb{Z}$. 
	Denote by $S$ the polynomial $S(t) = P(t+t_{0})$. We thus need to show that 
	\begin{equation*}
	\sum_{k=0}^{n} \binom{n}{k}^2  S(k) \geq 0 
	\; \textnormal{under the hypothesis that} \;
	S(t) \geq 0 \; \textnormal{for any} \; t \in \Lambda,
	\end{equation*}
	where $\Lambda := -t_{0}+\mathbb{Z}$. The set $\Lambda$ can be viewed as 
	$\Lambda =\{t\in\mathbb{R} \ |\ \cos(\pi t + \lambda) = 0\}$, where for instance 
	$\lambda:=\pi t_0+\pi/2$. Set $N(t) := t(t-1)(t-2)\ldots(t-n)$ and consider 
	the meromorphic function 
	\begin{equation*}
	F(z) =   \frac{ \tan(\pi z + \lambda) - \tan(\lambda)}{N(z)^2} S(z).
	\end{equation*}
	The poles of $F$ are simple and $F(z)$ decays like $|z|^{-2}$ on any large circle centered at zero   
	and does not pass through the poles of the function $\tan(\pi z + \lambda)$. 
	Therefore, the sum of residues of the function $F$ converges to zero. 
	The residue of $F$ at the zero $k$ of $N$ is given by
	\begin{equation*}
	Res_{z= k} F(z) = \frac{\pi}{(n!)^2 \cos^2(\lambda)} \binom{n}{k}^2 S(k),
	\quad k=0,1,\ldots,n,   
	\end{equation*}
	while the residue of $F$ at a pole $t \in \Lambda$ is given by 
	\begin{equation*}
	Res_{z= t \in \Lambda} F(z) = -\frac{S(t)}{\pi N(t)^2}.
	\end{equation*}
	Thus we obtain 
	\begin{equation*}
	\sum_{k=0}^{n} \binom{n}{k}^2  S(k)  = 
	\frac{(n!)^2 \cos^2 \lambda}{\pi^2} \sum_{t \in \Lambda} 
	\frac{S(t)}{N^2(t)} \geq 0.
	\end{equation*}
	This concludes the proof.
\qed
\end{proof}

We are now in a position to give an upper bound for the quantity $\tau_{p}$ defined in (\ref{TauDefinition}). 

\begin{theorem}
	For any positive integer $p$, we have $\tau_{p} \leq p-1$.
\end{theorem}
\begin{proof}
	Let $R$ be a fixed real number and $P$ be an admissible polynomial of 
	degree $2p-2$. According to Proposition \ref{Conjecture}, the polynomial
	$Q(t) = \sum_{k=0}^{p-1} \binom{p-1}{k}^2 P(t+k)$ is non-negative over 
	the whole real line. Applying Gauss quadrature and taking in account the 
	non-negativity of $Q$, we obtain
	\begin{equation*} 
	\begin{split}
	\int t Q(t) d\mu_{R}(t) & = \sum_{k = 1}^{p} \alpha_k \lambda_{k,p}(R) Q(\lambda_{k,p}(R)) \\
	& \leq \lambda_{p,p}(R)  \sum_{k = 1}^{p} \alpha_k Q(\lambda_k(R)) = \lambda_{p,p}(R) \int  Q(t) d\mu_{R}(t),
	\end{split}
	\end{equation*} 
	where $\lambda_{1,p}(R)<\lambda_{2,p}(R)<\ldots<\lambda_{p,p}(R)$ are the zeros of the Poisson-Charlier 
	polynomial $C_{p}(.,R)$. Therefore, there exists an integer $j \leq p-1$ such that  
	\begin{equation*}
	\int t P(t+j) d\mu_{R}(t) \leq \lambda_{p,p}(R) \int  P(t+j) d\mu_{R}(t). 
	\end{equation*} 
	Using Corollary \ref{iterated1}, we obtain 
	\begin{equation*}
	\frac{\int t P(t) d\mu_{R}(t)}{\int P(t) d\mu_{R}(t)} \leq 
	\frac{\int t P(t+j) d\mu_{R}(t)}{\int P(t+j) d\mu_{R}(t)} +j 
	\leq \lambda_{p,p}(R) + j.
	\end{equation*}
	Therefore, $\tau_{p} \leq j \leq p-1$. This concludes the proof.
\qed
\end{proof}
We are now in a position to prove Theorem \ref{MainTheorem2} (see Introduction).

\smallskip 

\noindent
{\bf{Proof of Theorem \ref{MainTheorem2}: }}
Using the well known fact that the zeros of the generalized Laguerre polynomials
$L_n^{(\alpha)}$ are increasing function of the parameter $\alpha \in (-1,\infty)$ 
\cite[pp. 121-122]{szego} along with Theorem \ref{lowerbound1} complete 
the proof of Theorem \ref{MainTheorem2}. 

\begin{remark}
	Lower bounds for the quantity $\tau_{p}$ can be obtained using the stability property 
	of the optimal threshold factors along the diagonals. For example, it is shown in \cite{kraaPoly} 
	that $R_{5,3}$ is the unique positive real zero of the cubic equation $r^3 - 5 r^2 +10r -10 = 0$, i.e; 
	$R_{5,3} \simeq 2.6506$ and that $R_{n+2,n} = R_{5,3}$ for any $ n\geq 3$. Thus, from Theorem 
	\ref{lowerbound1}, for any $p \geq 2$
	\begin{equation}
	\label{InequalityTau}
	R_{2p+1,2p-1} = R_{5,3} \geq \ell_{p}^{(p+1-\tau_{p})}.
	\end{equation}
	Let $\alpha$ be the unique real number such that $\ell_{p}^{(\alpha)} = R_{5,3} \simeq 2.6506$. 
	The monotonicity of the zero of Laguerre polynomials with respect to 
	the parameter $\alpha \in (-1,\infty)$ enables us to conclude from (\ref{InequalityTau}) that 
	$\tau_{p} \geq p+1-\alpha$. Some of the upper bounds to $\tau_{p}$ using this inequality are
	\begin{equation*}
	\tau_{5} \geq 0.74, \quad \tau_{8} \geq 1.99, \quad \tau_{12} \geq 4.10, \quad \tau_{15} \geq 5.88, 
	\quad \tau_{22} \geq 10.42. 
	\end{equation*}    
\end{remark} 

\section{Structural properties of the optimal threshold polynomials}
In this section, we adapt and extend an ingenious technique by Bernstein \cite{bernstein} 
to identify a structural property of the optimal threshold polynomial $\Phi_{m,n}$ that will
be fundamental throughout the rest of the paper. To ease our exposition we adopt the following
terminology. When we write a polynomial in the form 
\begin{equation}
\label{Model}
\Phi(x) = \sum_{k=1}^{s} \alpha_k \left(1 + \frac{x}{R}\right)^{m_k}
\; \textnormal{with} \; m_i \not= m_j, \; \textnormal{if} \; i\not=j;\; 
i,j=1,2,\ldots,s,  
\end{equation}
then we will call the integers $m_1,m_2,\ldots,m_s$ the {\it{exponents}} of 
$\Phi$ and the real numbers $\alpha_1,\alpha_2,\ldots,\alpha_s$ the {\it{coefficients}} 
of $\Phi$. For a given index $k$, we shall call $\alpha_k$ the coefficient associated with 
$m_k$ or simply the coefficient of $m_k$. We will use the term {\bf{missing exponents}} 
for exponents $m_k$ whose associated coefficients $\alpha_k$ are equal to zero. If an 
exponent $m_k$ is missing in the representation (\ref{Model}) then it is in fact 
a virtual exponent and can be placed anywhere at will. Thus when we say that the sequence 
of exponents $(m_1,m_2,\ldots,m_s)$ satisfies a certain property $(P)$ we mean that we can find positions
for the missing exponents such that the resulting sequence of exponents satisfies the property $(P)$. 
We use the expression {\bf{explicitly missing exponent}} to refers to the fact that we have deleted 
the missing exponent from the exponents sequence. For instance, when we say that a finite sequence 
$(m_1,m_2,m_3,m_4)$ is given by $(1,4,5)$ then necessarily there is one explicitly missing exponent.    
We have purposely avoided the use of the terminology of {\it{principal polynomials}} as in the seminal work of Bernstein \cite{bernstein} 
for the following reason: In Bernstein work, the sequence of exponents is not bounded above,
while in our case all the exponents of the optimal threshold polynomial $\Phi_{m,n}$ are at most equal to $m$. 
It will be also helpful to explicitly state the following simple theorem showing that there are 
four different ways to look at the problem at hand. The proof being implicitly contained in the previous sections,
we leave it to the reader.

\begin{theorem}
	\label{EquivalenceTheorem}
	Let $m, n$ be positive integers such that $m \geq n$. 
	Let $(\alpha_1,\alpha_2,\ldots,\alpha_s)$ be non-negative real numbers and 
	$(m_1,m_2,\ldots,m_s)$ be pair-wise distinct non-negative integers. 
	The following statements are equivalent. 
	
	\begin{enumerate}[\quad\rm(\roman{enumi})]
		\item The polynomial $\mathcal{H}_{n}(x,.)$ can be written as 
		\begin{equation}
		\label{NonSystem1}
		\mathcal{H}_{n}(x,R) = \int (x-t)^{n} d\mu_{R}(t) = \sum_{k=1}^{s} \alpha_{k} (x-m_{k})^n.
		\end{equation}
		\item The real numbers $\alpha_{1},\ldots,\alpha_{s}$ and the integers $m_{1},\ldots,m_{s}$ 
		satisfy the system 
		\begin{equation}
		\label{YesSystem}
		\sum_{k=1}^{s} \alpha_{k} m_{k}^{\ell} = B_{\ell}(R) 
		\quad \textnormal{for} \quad \ell=0,1,\ldots,n.
		\end{equation}
		\item The measure $\mu_{R}$ possesses a positive quadrature with integer nodes,\ie 
		for any polynomial $P$ of degree at most $n$, we have  
		\begin{equation*}      
		\int P(t) d\mu_{R}(t) = \sum_{k=1}^{s} \alpha_{k} P(m_{k}).
		\end{equation*}
		\item The polynomial $\Phi$ defined by 
		\begin{equation}
		\label{NonSystem2}
		\Phi(x) = \sum_{k=1}^{s} \alpha_k \left( 1 + \frac{x}{R}\right)^{m_k}
		\end{equation} 
		is absolutely monotonic over the interval $[-R,0]$ and it satisfies 
		$\Phi(x) - e^{x} = \mathcal{O}(x^{n+1})$ as $x\rightarrow 0$.
	\end{enumerate} 
\end{theorem}
Sometimes we shall refer to (\ref{NonSystem1}) or to (\ref{NonSystem2}) as being a system in $\alpha_i, m_i, i=1,\ldots,s$ 
when we actually mean the system (\ref{YesSystem}). The following fundamental theorem is based on ideas 
by Bernstein in \cite{bernstein}.  
\begin{theorem}
	\label{fundamental-theorem2}
	For any positive integers $m$ and $p$ such that $m \geq 2p-1$, 
	the optimal threshold polynomial $\Phi_{m,2p-1}$ has the form 
	\begin{equation}
	\label{structureP}
	\Phi_{m,2p-1}(x) = \sum_{k=1}^{2p} \alpha_k \left(1 + \frac{x}{R_{m,2p-1}}\right)^{m_k},
	\end{equation}
	where $\alpha_k, k=1,2,\ldots,2p$ are non-negative real numbers and where 
	the integers $0 \leq m_1<m_2<\ldots<m_{2p-1}<m_{2p} \leq m$ can be grouped in the form 
	\begin{equation}
	\label{myconditionm1}
	(q_1,q_1+1), (q_2,q_2+1),\ldots, (q_{p-1},q_{p-1}+1), q_{p}
	\end{equation} 
	with one explicitly missing exponent (and possibly other missing exponents) or of the form 
	\begin{equation}
	\label{myconditionm2}
	(q_1,q_1+1), (q_2,q_2+1),\ldots, (q_{p-1},q_{p-1}+1), (q_{p},q_{p}+1)
	\end{equation}
	with at least one missing coefficient. In (\ref{myconditionm1}) and (\ref{myconditionm2}),
	the integers $q_{1},q_{2},\ldots,q_{p}$ satisfy the inequalities
	\begin{equation*}
	q_{k}+1 < q_{k+1}, \quad k=1,2,\ldots,p-1.   
	\end{equation*}
\end{theorem}
\begin{proof}
	The strategy of the proof consists in showing that, if in the representation (\ref{structureP}) 
	of the optimal threshold polynomial, the sequence of integers $(m_1,m_2,\ldots,$ $m_{2p-1},m_{2p})$ 
	satisfy none of the conditions (\ref{myconditionm1}) and (\ref{myconditionm2}) then,
	starting from this representation, we can construct another representation of the optimal threshold polynomial
	whose exponents satisfy either (\ref{myconditionm1}) or (\ref{myconditionm2}). This will eventually contradict
	the uniqueness of the optimal threshold polynomial and thus conclude 
	the proof of the theorem. According to Theorem \ref{EquivalenceTheorem}, equation (\ref{structureP}) is 
	equivalent to the linear system
	\begin{equation}
	\label{MyLinearSystem}
	\left\{
	\TABbinary\tabbedCenterstack[l]{
		\alpha_1 +& \alpha_2 +& \ldots +& \alpha_{2p}  &= B_0(R_{m,2p-1}) \\
		\alpha_1 m_1 +& \alpha_2 m_2 +& \ldots +& \alpha_{2p} m_{2p}  &=  B_1(R_{m,2p-1})\\
		\ldots & \ldots & \ldots &  &\ldots \\ 
		\alpha_1 m_1^{k} +& \alpha_2 m_2^{k} +& \ldots +& \alpha_{2p} m_{2p}^{k}  &=  B_k(R_{m,2p-1}) \\       
		\ldots & \ldots & \ldots &  &\ldots  \\ 
		\alpha_1 m_1^{2p-1} +& \alpha_2 m_2^{2p-1} +& \ldots +& \alpha_{2p} m_{2p}^{2p-1}  &=   B_{2p-1}(R_{m,2p-1}) 
	}\right.
	\end{equation}
	Without loss of generality, we assume that $\alpha_{2p} >0$. In the above linear system, let us fix  
	all the integers $m_{k}, k < 2p$ and change continuously the value of $m_{2p}$ viewed as a real number.
	The variation of the coefficients $\alpha_1,\alpha_2,\ldots,\alpha_{2p}$ satisfies the linear system 
	\begin{equation*}
	\left\{
	\TABbinary\tabbedCenterstack[l]{
		\frac {\partial \alpha_1}{\partial m_{2p}} + & \frac {\partial \alpha_2}{\partial m_{2p}} 
		+& \ldots +& \frac {\partial \alpha_{2p}}{\partial m_{2p}}  &= 0 \\
		& & & & \\
		m_1 \frac {\partial \alpha_1}{\partial m_{2p}} + & m_2 \frac {\partial \alpha_2}{\partial m_{2p}} 
		+& \ldots +& m_{2p} \frac {\partial \alpha_{2p}}{\partial m_{2p}}  &=  -\alpha_{2p} \\ 
		& \ldots & \ldots & \ldots &  \\        
		& \ldots & \ldots & \ldots &  \\ 
		m_1^{2p-1} \frac {\partial \alpha_1}{\partial m_{2p}} + & m_2^{2p-1} 
		\frac {\partial \alpha_2}{\partial m_{2p}} +& \ldots +& m_{2p}^{2p-1} 
		\frac {\partial \alpha_{2p}}{\partial m_{2p}}  &=  -(2m-1)\alpha_{2p}m_{2p}^{2p-2}. 
	}\right.
	\end{equation*}
	The solution to the above linear system is given by 
	\begin{equation}
	\label{PartialAlpha}
	\frac {\partial \alpha_k}{\partial m_{2p}} = -\frac{\alpha_{2p}}{\Delta} 
	\frac{\partial \Delta_{m_k}}{\partial m}(m_{2p}), \quad k=1,2,\ldots,2p, 
	\end{equation}
	where $\Delta  = \prod_{1 \leq i < j \leq 2p } (m_j - m_i)$ and $\Delta_{m_k}(m)$ is given 
	by the function determinant 
	\[
	\Delta_{m_k}(m) =
	\begin{vmatrix}
	1 & \ldots & 1 & 1 & 1 & \dots & 1 \\
	m_1 & \ldots & m_{k-1} & m & m_{k+1} &\dots & m_{2p}  \\ 
	\ldots & \ldots & \ldots & \ldots &\dots & \ldots  \\
	m_1^{2p-1} & \ldots & m_{k-1}^{2p-1} & m^{2p-1} & m_{k+1}^{2p-1}&  \dots & m_{2p}^{2p-1}
	\end{vmatrix}
	\]
	For $k<2p$, the largest zero of $\Delta_{m_k}$ is $m_{2p}$. 
	Thus the sign of $\frac{\partial \Delta_{m_k}}{\partial m}(m_{2p})$ 
	is the same as the sign of $\Delta_{m_k}(m)$ for $m > m_{2p}$. 
	Thus this sign is positive for even $k$ and negative for odd $k$.
	Therefore, we conclude from (\ref{PartialAlpha}) that 
	\begin{equation}
	\label{signAlpha}
	(-1)^k \frac {\partial \alpha_k}{\partial m_{2p}} < 0, \quad k=1,2,\ldots,2p.
	\end{equation} 
	From (\ref{signAlpha}) we infer that if we increase the value of $m_{2p}$ 
	and solve the corresponding linear system (\ref{MyLinearSystem}) then all the 
	coefficients with odd index $\alpha_{2k-1}, k=1,\ldots,p$, will increase, 
	while the coefficients with even index $\alpha_{2k}, k=1,\ldots,p$, will decrease. 
	The opposite happens if we proceed by decreasing the value of $m_{2p}$.
	Now, assuming that the exponents of the decomposition (\ref{structureP}) satisfy neither 
	(\ref{myconditionm1}) nor (\ref{myconditionm2}), consider the associated system (\ref{MyLinearSystem}).
	We start a {\it{descending process}} by decreasing the value of $m_{2p}$ while avoiding 
	that any of the coefficients $\alpha_k, k=1,2,\ldots,2p-1$, obtained by solving 
	(\ref{MyLinearSystem}), becomes negative. Noting that the missing exponents $m_i$ from (\ref{structureP}) are virtual 
	and can be placed anywhere at will, we can easily deduce that a decrease of $m_{2p}$ is impossible only if 
	the exponents $(m_1,m_2,\ldots,m_{2p-1})$ can be grouped into integers of the form 
	\begin{equation}
	\label{MainGroup}
	(q_1,q_1+1), (q_2,q_2+1), \ldots (q_{p-1},q_{p-1}+1), \quad q_{k}+1 < q_{k+1}, \quad k=1,2,\ldots,p-2,   
	\end{equation}  
	with one explicitly missing exponent (and possibly other missing exponents). 
	From our hypothesis, such a decrease of $m_{2p}$ is then possible. Thus, we decrease the value of $m_{2p}$ until one of the odd coefficients $\alpha_{2k-1}$ vanishes. 
	This eventually happens before the value of $\alpha_{2p}$ vanishes
	due to the fact that if $\alpha_{2p}=0$ before any of the odd coefficients vanishes 
	then it should have been zero before the start of the descending process. 
	Thus once one of the odd coefficients $\alpha_{2k-1}$ vanishes, we replace the corresponding virtual exponent 
	$m_{2k-1}$ by the largest integer $q < m_{2p}$ such that there are an odd number of integers $m_i$ between $q$ and $m_{2p}$ 
	(assuming, for the moment, that such move is possible). Note that a further decrease of $m_{2p}$ will now increase the new value of $\alpha_{2k-1}$ 
	as the index of its corresponding exponent is now even. We continue this descending process until no further decrease 
	of $m_{2p}$ is possible. This is the case only when the exponents $(m_{1},m_{2},\ldots,m_{2p-2})$ can be grouped into integers
	of the form (\ref{MainGroup}). If at the end of the descending process, the real number $m_{2p}$ is an integer 
	then we have found a solution to our linear system (\ref{MyLinearSystem}) where the exponents satisfy 
	condition (\ref{myconditionm1}) and the associated coefficients are non-negative. This contradicts the uniqueness 
	of the optimal threshold polynomial. Let us assume now that at the end of the descending process, the real number 
	$\rho := m_{2p}$ is not an integer. Since, there is a least one missing exponent at the end of the descending process,
	we place this missing exponent at the position of the largest integer $q < m_{2p}$ that is not occupied by another 
	exponent with positive coefficient. We have then a configuration of exponents of the form
	\begin{equation}
	\label{newexponent}
	(q_1,q_1+1), (q_{2},q_{2}+1) \ldots (q_{p-1},q_{p-1}+1), (\bar{q},\rho), 
	\end{equation}           
	where $\bar{q}$ is an integer and $\bar{q} < \rho < \bar{q}+1$. 
	Moreover, the coefficients associated with even index exponents are non-zero.   
	Now we start an {\it{ascending process}} by increasing the value of $\rho$ while solving the
	corresponding linear system (\ref{MyLinearSystem}). In doing so, the coefficients with even index will 
	decrease, while the one with odd index will increase. If we increase $\rho$ until $\bar{q}+1$ without any 
	of the coefficients $\alpha_k$ becomes negative then we would arrive to a solution of the linear system (\ref{MyLinearSystem}) 
	where the exponents satisfy condition (\ref{myconditionm2}) and thus again in contradiction with the uniqueness of the optimal
	threshold polynomial. If during the ascending process of $\rho$ and before $\rho$ reaches $\bar{q}+1$, one of the even 
	index coefficients $\alpha_{2k}$ vanishes, then we replace the associated exponent $q_{k}+1$ by $q_{k}-1$ and continue 
	the ascending process. However, if the site $q_{k}-1$ is already occupied by another exponent with positive 
	coefficient, then we replace it by $q_{k-2}-1$ instead and so on, and then continue the ascending process. 
	The only case where an increase of $\rho$ is no longer possible is when the exponents are grouped into pairs of the form            
	\begin{equation}
	\label{forbidenCon}
	(0,1),(2,3),\ldots,(2i-2,2i-1),(q_{i+1},q_{i+1}+1) \ldots (q_{p-1},q_{p-1}+1),(\bar{q},\rho),
	\end{equation}      
	with the coefficient of one exponent among $(1,3,5,\ldots,2i-1)$, say $k$, equal to zero.
	Let us show that a configuration such as (\ref{forbidenCon}) cannot be reached before 
	$\rho$ reaches $\bar{q}+1$. Let us consider the configuration (\ref{forbidenCon}) with $\rho < \bar{q} + 1$. 
	The associated polynomial
	\begin{equation*}
	\psi(t) =  (t-\bar{q}) (t - \rho) \prod_{j=0,j\not=k}^{2i-1} (t - j) \prod_{j=i+1}^{p-1} (t - q_{j})(t-q_{j}-1)
	\end{equation*} 
	is non-negative on $\mathbb{N}$ and thus we have 
	\begin{equation}
	\label{ZeroIntegral}
	\int \psi(t) d\mu_{R_{m,2p-1}}(t) > 0.
	\end{equation}
	For the configuration (\ref{forbidenCon}), we have 
	\begin{equation}
	\label{tempProof}
	\mathcal{H}_{2p-1}(x;R_{m,2p-1}) = \int (x-t)^{2p-1} d\mu_{R_{m,2p-1}}(t) = \sum_{j=1}^{2p-1} \alpha_j (x - x_{j})^{2p-1},
	\end{equation}
	where the $x_j$'s are the zeros of the polynomial $\psi$ and $\alpha_1,\alpha_2,\ldots,\alpha_{2p-1}$ are non-negative numbers. 
	Evaluating the polar form of both sides of (\ref{tempProof}), we obtain  
	\begin{equation*}
	\int \psi(t) d\mu_{R_{m,2p-1}}(t)  = h_{2p-1}(0,1,\ldots,k-1,k+1,\ldots,q_{p-1}+1,\bar{q},\rho;R_{m,2p-1}) = 0.
	\end{equation*} 
	This contradicts (\ref{ZeroIntegral}). Therefore a configuration of the form 
	(\ref{forbidenCon}) cannot be reached before $\rho$ reaches $\bar{q}+1$. Hence, once $\rho$ reaches $\bar{q}+1$, 
	we obtain a solution of our linear system (\ref{MyLinearSystem}) where the exponents satisfy condition 
	(\ref{myconditionm2}) and where the associated coefficients are non-negative. 
	This again contradicts the uniqueness of the optimal threshold polynomial. This concludes the proof of the theorem. 
\qed 
\end{proof}

\begin{remark}
In Theorem \ref{IntegerCharlier}, we have shown that the exponents of the optimal threshold polynomial 
$\Phi_{m,3}$, with $m$ is a square integer are $m-2\sqrt{m} +1, m $. These exponents can be grouped 
in the form (\ref{myconditionm1}) as 
\begin{equation*}
(m-2\sqrt{m} +1,m-2\sqrt{m} +2),m,
\end{equation*}
where $m-2\sqrt{m} +2$ is a missing exponent. Another interesting example is the optimal threshold polynomial $\Phi_{6,5}$ \cite{[higueras]}. We have $R_{6,5} = 2$ and 
\begin{equation}
\label{hig}
\Phi_{6,5}(x) = \frac{1}{6} + \frac{2}{5}\left(1 + \frac{x}{2}\right) + \frac{4}{9}\left(1 + \frac{x}{2}\right)^{3} 
+ \frac{2}{45}\left(1 + \frac{x}{2}\right)^{6}.
\end{equation}
The exponents in (\ref{hig}) are $0,1,3,6$ and can be grouped in the form  
\begin{equation*}
(0,1),(3,4),6,
\end{equation*}
where $4$ is a missing exponent.

\end{remark}

The representations (\ref{myconditionm1}) and (\ref{myconditionm2}) of the exponents of 
the optimal threshold polynomial are quite similar. For example, if the coefficient 
of $q_{p}$ (or of $q_{p}+1$) is equal to zero in (\ref{myconditionm2}) then the two 
representations coincide. However, for instance a representation where the exponents 
are grouped in the form
\begin{equation}
\label{ExampleConfig}
(2,3),(5,6),(10,11),(12,13) 
\end{equation} 
where the coefficients of the exponents $10,11,12,13$ are positive (and at least one of 
the exponents $2,3,5,6$ is missing) cannot be represented in the from (\ref{myconditionm1}). 
Nevertheless, if the coefficients of $11,12,13$ are positive while the coefficient of $10$
is equal to zero, \ie $10$ is a missing exponent, then we can re-write (\ref{ExampleConfig}) 
in the form (\ref{myconditionm1}) via a shift in the indices,\ie 
\begin{equation*}
(2,3),(5,6),(11,12), 13. 
\end{equation*} 
with $10$ is an explicitly missing exponent. According to this simple observation, we shall provide
a more refined structural property of the optimal threshold polynomial 
$\Phi_{m,2p-1}$ by showing that its exponents can always be represented in the form (\ref{myconditionm1}) 
with $q_{p}=m$. For this we shall need the following proposition.

\begin{proposition}
	\label{NonVanishPair}
	The coefficients of each pair, in the two possible representations (\ref{myconditionm1}) and (\ref{myconditionm2}) 
	of the exponents of the optimal threshold polynomial $\Phi_{m,2p-1}$, cannot be simultaneously equal to zero.
\end{proposition}
\begin{proof}
	We give a proof for representations of the exponents of the form (\ref{myconditionm1}). 
	Representations of the from (\ref{myconditionm2}) can be handled in a similar fashion. 
	Let us assume that the coefficients $\alpha_{2\ell-1}$ and $\alpha_{2\ell}$ ($\ell \leq p-1$) 
	associated with one of the pair $(q_{\ell},q_{\ell}+1)$ are both zero. Then we have 
	\begin{equation}
	\label{Pairs}
	\begin{split}
	& \mathcal{H}_{2p-1}(x; R_{m,2p-1}) = \int (x-t)^{2p-1} d\mu_{R_{m,2p-1}} \\
	& =\alpha_{2p-1} (x - q_{p})^{2p-1} +  \sum_{k =1,k\not=\ell}^{p-1} 
	\alpha_{2k-1}(x - q_{k})^{2p-1} +
	\alpha_{2k}(x - q_{k}-1)^{2p-1}. 
	\end{split}
	\end{equation}  
	Differentiating (\ref{Pairs}) with respect to $x$, we obtain
	\begin{equation}
	\label{Pairs2}
	\begin{split}
	&\mathcal{H}_{2p-2}(x; R_{m,2p-1})   = \int (x-t)^{2p-2} d\mu_{R_{m,2p-1}} \\
	& = \alpha_{2p-1} (x - q_{p})^{2p-2} +  \sum_{k =1,k\not=\ell}^{p-1} 
	\alpha_{2k-1}(x - q_{k})^{2p-2} +
	\alpha_{2k}(x - q_{k}-1)^{2p-2}.
	\end{split}
	\end{equation}
	Denote by $h_{2p-2}(-; R_{m,2p-1})$ the polar form of the polynomial $\mathcal{H}_{2p-2}(.; R_{m,2p-1})$ and 
	$\Lambda = (q_1,q_1+1,\ldots,q_{\ell-1},q_{\ell-1}+1,,q_{\ell+1},q_{\ell+1}+1,\ldots q_{p},q_{p}+1)$. 
	From (\ref{Pairs2}) we obtain the following contradiction 
	\begin{equation*}
	h_{2p-2}(\Lambda;R_{m,2p-1}) = 0 = \int \prod_{k=1,k\not=\ell}^{p} (q_{k} - t)(q_{k}+1 -t) d\mu_{R_{m,2p-1}} >0. 
	\end{equation*}
	This completes the proof.  
\qed
\end{proof}

We are now in a position to give a refined structural property of the optimal threshold polynomial.

\begin{theorem}
	\label{fundamental-theorem3}
	For any positive integers $m$ and $p$ such that $m \geq 2p-1$, 
	the optimal threshold polynomial $\Phi_{m,2p-1}$ has the form 
	\begin{equation}
	\label{structurePRefined2}
	\Phi_{m,2p-1}(x) = \sum_{k=1}^{2p} \alpha_k \left(1 + \frac{x}{R_{m,2p-1}}\right)^{m_k},
	\end{equation}
	where $\alpha_k, k=1,2,\ldots,2p$ are non-negative real numbers and the integers 
	$0\leq m_1<m_2<\ldots<m_{2p-1}\leq m$ can be grouped into the form  
	\begin{equation}
	\label{MainGroup7}
	(q_1,q_1+1), (q_2,q_2+1),\ldots (q_{p-1},q_{p-1}+1), q_p; \; \; q_{k}+1 < q_{k+1}, \; k=1,2,\ldots,p-1,   
	\end{equation} 
	with one explicitly missing exponent (and possibly many missing exponents). Moreover, we have
	$q_{p} = m$ and its coefficient is positive.
\end{theorem} 
\begin{proof}
	From Theorem \ref{fundamental-theorem2}, we know that the exponents $(m_1,m_2,\ldots,m_{2p})$ 
	can be represented in the form (\ref{myconditionm1}) or (\ref{myconditionm2}).  
	Assume that a configuration of exponents of $\Phi_{m,2p-1}$ of the form 
	\begin{equation}
	\label{MainGroup2}
	\begin{split}
	& (q_1,q_1+1),\ldots,(q_k,q_k+1),(q_{k+1},q_{k+1}+1), \\
	& (q_{k+1}+2,q_{k+1}+3),\ldots,(q_{k+1}+2s,q_{k+1}+2s+1), 
	\end{split}
	\end{equation}
	with $k \leq p-1$, $k+1+s = p$ and where the coefficients associated with $q_{k+1},q_{k+1}+1,\ldots,q_{k+1}+2s+1$
	are positive with $q_{k+1} - (q_{k}+1) >1$ is possible. Then we can 
	always assume that the coefficients associated with the first element of each pair in (\ref{MainGroup2}) 
	is positive using the following procedure: If, for example, the coefficient associated with the first element 
	of a pair $(q_{\ell},q_{\ell}+1)$ ($\ell \leq k$) is equal to zero, then according to Proposition (\ref{NonVanishPair}), 
	the coefficient associated with $q_{\ell}+1$ is positive. 
	In this case we change the pair $(q_{\ell},q_{\ell}+1)$ into the pair $(q_{\ell}+1,q_{\ell}+2)$.
	If the site $q_{\ell}+2$ is already occupied by an exponent with positive coefficient then we place $q_{\ell}$ at $q_{\ell}+4$ 
	instead and so on. The fact that we have assumed the existence of a least one free site between $q_{k}+1$ and $q_{k+1}$, i.e; 
	$q_{k+1} - (q_{k}+1) >1$, insures the success of such procedure. Now that the coefficients associated with the first element 
	of each pair in (\ref{MainGroup2}) is positive, we start decreasing the value of $q_{k+1}+2s+1$ and solve the associated 
	linear system (\ref{MyLinearSystem}). As we have shown before, a decrease of $q_{k+1}+2s+1$ will increase the coefficients 
	with even index and decrease the ones with odd index. Therefore, a small decrease of $q_{k+1}+2s+1$ say to $q_{k+1}+2s+1 - \delta$ 
	will render all the coefficients $\alpha_i$ positive. At this stage, we increase the value of $R_{m,2p-1}$ to $R_{m,2p-1}+\epsilon$ in such 
	a way that all the coefficients $\alpha_i$ solution to the new linear system (\ref{MyLinearSystem}) remain positive. 
	and then we bring $q_{k+1}+2s+1-\delta$ to $q_{k+1}+2s+1$ by the same ascending process as in the proof of Theorem (\ref{fundamental-theorem2}). 
	At the end of this procedure, we obtain a solution to the linear system (\ref{MyLinearSystem}) with non-negative coefficients $\alpha_i$ 
	and with $R_{m,2p-1}$ replaced by $R_{m,2p-1}+\epsilon$. This contradicts the very definition of $R_{m,2p-1}$. 
	Therefore, the only possible configurations of the exponents of the optimal threshold polynomial are the ones that are 
	of the form (\ref{myconditionm1}) or of the form (\ref{MainGroup2})
	with $k \leq p-1$,  $k+1+s = p$, but now the coefficient associated with $q_{k+1}$ 
	must be equal to zero while the coefficients associated with $q_{k+1}+1,q_{k+1}+2,\ldots,q_{k+1}+2s+1$ 
	must be positive. The latter configurations can be written in the form (\ref{myconditionm1}) 
	by a single shift of the indices as 
	\begin{equation*}
	\begin{split}
	& (q_1,q_1+1),\ldots,(q_k,q_k+1),(q_{k+1}+1,q_{k+1}+2), \\
	& (q_{k+1}+3,q_{k+1}+4),\ldots,(q_{k+1}+2s-1,q_{k+1}+2s),q_{k+1}+2s.  
	\end{split}
	\end{equation*}
	This proves the first part of the theorem. 
	Let us now prove that in the representation (\ref{MainGroup7}) we have $q_{p} = m$ and that 
	its associated coefficient is positive. According to what we have just proved, the optimal threshold 
	polynomial has the form 
	\begin{equation}
	\label{original}
	\Phi_{m,2p-1}(x) = \sum_{k=1}^{2p-1} \alpha_k \left(1 + \frac{x}{R_{m,2p-1}}\right)^{m_k},
	\end{equation}
	where $(m_1,m_2,\ldots,m_{2p-1}) = (q_1,q_1+1,q_2,q_2+1,\ldots,q_{p-1}+1,q_{p})$.  
	We proceed by contradiction and assume that $m_{2p-1} < m$. According to Theorem \ref{EquivalenceTheorem}, 
	the identity (\ref{original}) is equivalent to 
	\begin{equation*}
	\mathcal{H}_{2p-1}(x;R) = \int (x-t)^{2p-1} d\mu_{R_{m,2p-1}}(t) =  \sum_{k=1}^{2p-1} \alpha_k (x-m_k)^{2p-1}.
	\end{equation*}
	Consider the polynomial $F(x) := \mathcal{H}_{2p}(x,R_{m,2p-1}) = \int (x-t)^{2p} d\mu_{R_{m,2p-1}}(t)$ and denote by 
	$f$ its polar form. We have 
	\begin{equation}
	\label{polarF1}
	\begin{split}
	f(0,x^{[2p-1]}) & = -R_{m,2p-1} \int (x-1-t)^{2p-1} d\mu_{R_{m,2p-1}}(t) \\
	& = -R_{m,2p-1} \sum_{k=1}^{2p-1} \alpha_k (x-(m_k+1))^{2p-1}.
	\end{split}
	\end{equation}
	Hence $f(0,m_1+1,m_2+1,\ldots,m_{2p-1}+1)=0$ and thus according to 
	Proposition \ref{Prop:Polar}, there exist coefficients 
	$\beta_0,\beta_1,\ldots,\beta_{2p-1}$ such that 
	\begin{equation}
	\label{Fpoly1}
	F(x) = \beta_0 x^{2p} + \sum_{k=1}^{2p-1} \beta_k (x - (m_k+1))^{2p}.
	\end{equation}
	Computing $f(0,x^{[2p-1]})$ from (\ref{Fpoly1}) and comparing with (\ref{polarF1}) yields
	\begin{equation*}
	R_{m,2p-1} \sum_{k=1}^{2p-1} \alpha_k (x - (m_k+1))^{2p-1}  =  
	\sum_{k=1}^{2p} \beta_k (x - (m_k+1))^{2p-1}.
	\end{equation*} 
	Thus $\beta_k = \alpha_k \frac{R_{m,2p-1}}{m_k+1} \geq 0$ for $k=1,2,\ldots,2p$.
	Now we prove that $\beta_0 \geq 0$ as follows:
	\begin{equation*}
	\begin{split}
	& f(m_1+1,m_2+1,\ldots,m_{2p-1}+1,m_{2p-1}+2)  = \beta_0 (m_{2p-1} +2) \prod_{k=1}^{2p-1}(m_k+1) \\
	& = \int (m_{2p-1} +2 - t) \prod_{k=1}^{2p-1} (m_k-t) d\mu_{R_{m,2p-1}}(t) \geq 0 
	\end{split}
	\end{equation*}
	since according to the first part of the theorem, the polynomial 
	\begin{equation*}
	P(t) = (m_{2p-1} +2 - t) \prod_{k=1}^{2p-1} (m_k+1 -t)
	\end{equation*}
	satisfies $P(j) \geq 0$ for any $j \in \mathbb{N}$. Therefore, $\beta_0 \geq 0$. Moreover, we have 
	\begin{equation*}
	F'(x) = 2p \mathcal{H}_{2p-1}(x;R_{m,2p-1}) = 2p \left( \beta_{0} x^{2p-1} + \sum_{k=1}^{2p-1} (x - (m_{k}+1))^{2p-1} \right). 
	\end{equation*}
	The last identity is equivalent to 
	\begin{equation}
	\label{ContradictRepresentation}
	\Phi_{m,2p-1}(x) = \beta_0 + \sum_{k=1}^{2p-1} \beta_k \left(1 + \frac{x}{R_{m,2p-1}}\right)^{m_k+1}.
	\end{equation}
	Since $\beta_j \geq 0, j=0,\ldots,2p$ and $m_{j} +1 \leq m$ for $j=1,\ldots,2p$, the representation 
	(\ref{ContradictRepresentation}) contradicts the uniqueness of the optimal threshold polynomial. 
	Thus, we conclude that $m_{2p-1} = m$ and $\alpha_{2m-1} > 0$. 
\qed 
\end{proof}

We will find it convenient to re-write Theorem \ref{fundamental-theorem3} in the 
equivalent form stated in Theorem \ref{fundamental-corollary} (See Introduction).

\begin{example}
	Using the algorithm of Section 8, it can be shown that $R_{200,5} \simeq 175.8348$ is the unique 
	positive zero of the polynomial equation
	\begin{equation*}
	R^5 - 852R^4 + 291352R^3 - 49988400R^2 + 4303437600R - 148719648000 = 0.
	\end{equation*} 
	The optimal threshold polynomial $\Phi_{200,5}$ is given by (\ref{introth}) with
	\begin{equation*} 
	(m_1,m_2,m_3,m_4,m_5) = (154,155,176,177,200)
	\end{equation*}
	and $(\alpha_1,\alpha_2,\alpha_3,\alpha_4,\alpha_5) = (0.1846,0.0007,0.3320,0.3336,0.1491)$. 
	The structural property of $\Phi_{200,5}$ confirms the statement given in Theorem \ref{fundamental-corollary}. 
\end{example}

Theorem \ref{fundamental-corollary} shows in particular that for any positive 
integer $n$, $R_{m,n}$ is a strictly increasing function of the parameter $m \geq n$. We 
should also point out that the findings in Theorem \ref{fundamental-corollary} can be used 
to significantly improve Step $1$ in Kraaijevanger's algorithm as it considerably reduces 
the number of integer sequences to be considered in Step 1 of the algorithm.

As a first application of Theorem \ref{fundamental-theorem2},
we prove Theorem \ref{MainTheorem3} (See Introduction)

\smallskip 

\noindent
{\bf{Proof of Theorem \ref{MainTheorem3}: }} 
To prove that $R_{m+1,2p} \leq R_{m,2p-1}$ we proceed as follows: Let 
the optimal threshold polynomial $\Phi_{m+1,2p}$ be written as 
\begin{equation}
\label{optimalm1}
\Phi_{m+1,2p}(x) = \sum_{k=1}^{2p} \beta_k \left(1 + \frac{x}{R_{m+1,2p}}\right)^{m_k},
\end{equation}
with 
$\beta_k \geq 0$ for $k=1,2,\ldots,2p$ and 
$0 \leq m_1 < m_2 < \ldots < m_{2p} \leq m+1$. Since $\Phi_{m+1,2p}$ belongs to $\Pi_{m+1,2p}$,
its derivative with respect to $x$ belongs to $\Pi_{m,2p-1}$. 
Moreover, taking the derivative of (\ref{optimalm1}), we obtain 
\begin{equation*}
\frac{\partial \Phi_{m+1,2p} }{\partial x}(x) = \sum_{k=1}^{2p} \frac{m_k \beta_k}{R_{m+1,2p}}
\left(1 + \frac{x}{R_{m+1,2p}}\right)^{m_k-1}.
\end{equation*}   
Accordingly, Corollary \ref{HR} enables us to conclude that $R_{m+1,2p} \leq R_{m,2p-1}$.
To prove that $R_{m+1,2p} \geq R_{m,2p-1}$ we proceed exactly as in the proof of the second part of 
Theorem \ref{fundamental-theorem3}. Namely, from the optimal threshold polynomial 
\begin{equation*}
\Phi_{m,2p-1}(x) = \sum_{k=1}^{2p-1} \alpha_k \left(1 + \frac{x}{R_{m,2p-1}}\right)^{m_k},
\end{equation*}
we construct the polynomial $F$ in (\ref{Fpoly1}) with the properties
\begin{equation*}
F(x) := \mathcal{H}_{2p}(x,R_{m,2p-1}) = \int (x-t)^{2p} d\mu_{R_{m,2p-1}}(t) =  
\beta_0 x^{2p} + \sum_{k=1}^{2p-1} \beta_k (x - (m_k+1))^{2p},
\end{equation*}
with $\beta_{k} \geq 0, k=0,1,\ldots,2p$. The last identity shows, 
according to Corollary \ref{HR}, that $R_{m+1,2p} \geq R_{m,2p-1}$. 
The relation (\ref{OptimalRelation}) between the optimal threshold polynomials 
is a direct and simple consequence of the equality $R_{m,2p-1} = R_{m+1,2p}$.
That $R_{m,n}$ are algebraic numbers is a consequence of the fact that 
$R_{m,2p-1}$ is a zero of the polynomial equation in $R$ with integer coefficients 
\begin{equation*}
h_{2p-1}(m_1,m_2,\ldots,m_{2p-2},m;R) = 0,
\end{equation*} 
where $h_{2p-1}(-;R)$ the polar form of the polynomial $\mathcal{H}_{2p-1}(.;R)$.  
\QEDB 

\section{Spectral transformations and optimal threshold factors}
In this section we use the structural property stated in Theorem\ref{fundamental-corollary} 
to gain more insights on the optimal threshold polynomial. This will lead to a highly efficient 
algorithm for the computation of the optimal threshold factors and their associated polynomials.
From now on, we adopt the following notation and terminology. 
For an admissible polynomial $\Omega$, we define the {\it{Christoffel transform measure}} 
$\mu^{\Omega}_{R}$ of $\mu_{R}$  by   
\begin{equation*}
\mu_{R}^{\Omega} = e^{-R} \sum_{j=0}^{\infty} 
\Omega(j) \frac{R^j}{j!}\delta_j.
\end{equation*}
The polynomial $\Omega$ is called the {\it{annihilator polynomial}} of the measure $\mu^{\Omega}_{R}$.
The orthogonal polynomials with respect to the measure $\mu_{R}^{\Omega}$ are denoted by 
$\Pi_{n}^{R,\Omega}, n\geq 0$ and their zeros by
$\lambda_{1,n}^{\Omega}(R) < \lambda_{2,n}^{\Omega}(R)<\ldots<\lambda_{n,n}^{\Omega}(R)$ or simply 
$\lambda_{1,n}^{\Omega} < \lambda_{2,n}^{\Omega}<\ldots<\lambda_{n,n}^{\Omega}$ 
if the real number $R$ is understood within the context.

We shall need the following theorem by Sylvester \cite{syslvester,polya} 

\begin{theorem}(Sylvester). 
	Suppose $0 \not= \beta_k$ for all $k$ and 
	$\gamma_1 < \ldots < \gamma_r, r \geq 2$, are real numbers 
	such that
	\begin{equation}
	\label{SylvesterForm}
	Q(t) = \sum_{k=1}^r \beta_k(x - \gamma_k)^d
	\end{equation}
	does not vanish identically. Suppose the sequence 
	$(\beta_1,\ldots,\beta_r,(-1)^d \beta_1)$ has $C$ changes of 
	sign and $Q$ has $Z$ real zeros, counting multiplicities. Then $Z \leq C$.
\end{theorem}

Using Theorem \ref{fundamental-corollary} and Sylvester's theorem, we show the following. 

\begin{proposition}
	\label{SingleIteration}
	Let $\Phi_{m,2p-1}$ be the optimal threshold polynomial with optimal threshold factor $R_{m,2p-1}$
	\begin{equation*}
	\Phi_{m,2p-1}(x) = \sum_{k=1}^{2p-1} \alpha_k \left(1 + \frac{x}{R_{m,2p-1}}\right)^{m_k},
	\end{equation*}
	with $0 \leq m_1 < m_2 < \ldots < m_{2p-1} = m$. Let $\lambda_{1,p}<\lambda_{2,p}<\ldots<\lambda_{p,p}$ be the zeros 
	of the Poisson-Charlier polynomial $C(.,R_{m,2p-1})$. Then there exist an odd index $j_{1}\leq 2p-1$ and 
	an integer $k_{1}\leq p$ such that
	\begin{equation*}
	m_{j_{1}} = \floor{\lambda_{k_{1},p}} \quad \textnormal{and} \quad m_{j_{1}+1} = \floor{\lambda_{k_{1},p}}+1,
	\end{equation*} 
	where $\floor{x}$ refers to the greatest integer not exceeding $x$.
\end{proposition} 
\begin{proof}
	By Gauss quadrature, we have  
	\begin{equation}
	\label{representation1}
	\int (x-t)^{2p-1} d\mu_{R_{m,2p-1}} = \sum_{i=1}^{p} \omega_{i} (x-\lambda_{i,p})^{2p-1}; 
	\quad \omega_i > 0 \quad \textnormal{for} \quad i=1,2,\ldots,p.
	\end{equation} 
	Moreover, from Theorem \ref{introth}, we know that 
	\begin{equation}
	\label{representation3}
	\int (x-t)^{2p-1} d\mu_{R_{m,2p-1}} = \sum_{i=1}^{2p-1} \alpha_{i} (x-m_{i})^{2p-1},
	\end{equation}
	where the integer sequence $(m_1,m_2,\ldots,m_{2p-3},m_{2p-1})$ satisfies (\ref{introm}).  
	Let us assume that there exists an integer $k$ ($k\leq p-1$)  such that
	\begin{equation}
	\label{NoInterval}
	m_{i} \notin [\lambda_{k,p}, \lambda_{k+1,p}] \quad \textnormal{for any} \quad  i=1,2,\ldots,2p-1.
	\end{equation} 
	Let $j \leq p$ be a positive integer different from $k$. From (\ref{representation1}) 
	and (\ref{representation2}) 
	\begin{equation}
	\label{ApplySylvester}
	\sum_{i=1}^{2p-1} \alpha_{i} (x-m_{i})^{2p-1}- \sum_{i=1,i\not=j}^{p} 
	\omega_{i} (x-\lambda_{i,p})^{2p-1} = \omega_{j} (x-\lambda_{j,p})^{2p-1}. 
	\end{equation} 
	Eliminating the zero coefficients $\alpha_i$ (if any) from the left-hand side 
	of (\ref{ApplySylvester}) and in account of (\ref{NoInterval}), we can easily show that,
	no matter how we place the integers $m_i$ relative to the real numbers $\lambda_{j,p}$ and write 
	the left-hand side of (\ref{ApplySylvester}) in the form (\ref{SylvesterForm}), the number of changes of sign 
	of the obtained sequence $(\beta_1,\ldots,\beta_r,(-1)^{2p-1} \beta_1)$ is less than $2p-1$.
	This contradicts Sylvester's theorem since the number of zeros of the right-hand side of 
	(\ref{ApplySylvester}) is equal to $2p-1$, counting multiplicities. Thus,  
	between $\lambda_{k,p}$ and $\lambda_{k+1,p}$ for $k=1,2,\ldots,p-1$, there is 
	at least one integer from the sequence $(m_1,m_2,\ldots,m_{2p-3},m_{2p-2},m)$. 
	Furthermore, since $m_{1} \leq \lambda_{1}$, $\lambda_{p} \leq m$ 
	(See Proposition \ref{last-term}) and due to (\ref{introm}), we easily conclude 
	that there exist an odd index $j_{1}\leq 2p-1$ and an integer $k_{1}\leq p$ such that 
	$m_{j_{1}} \leq \lambda_{k_{1},p} \leq m_{j_{1}+1} =  m_{j_{1}}+1$. 
	If $\lambda_{k_{1},p} \not= m_{j_{1}+1}$ then  $m_{j_{1}} = \floor{\lambda_{k_{1},p}}$ and $m_{j_{1}+1} = \floor{\lambda_{k_{1},p}}+1$
	and the claim is proved. If  $\lambda_{k_{1},p} = m_{j_{1}+1}$, we take the Christoffel transform 
	$\mu^{\tilde{\Omega}}_{R_{m,2p-1}}$ of $\mu_{R_{m,2p-1}}$ where $\tilde{\Omega}(t) = (t-m_{j_{1}})(t - m_{j_{1}+1})$. 
	Evaluating the polar form to both sides of (\ref{representation1}) and (\ref{representation3}) 
	at $(x^{[2p-3]},m_{j_{1}},m_{j_{1}+1})$, we obtain  
	\begin{equation}
	\label{express1}
	\int (x-t)^{2p-3} d\mu^{{\tilde{\Omega}}}_{R_{m,2p-1}} = 
	\sum_{i=1, i\not= j_{1},j_{1}+1 }^{p-1} \tilde{\omega}_{i} (x-\lambda_{i,p})^{2p-3},
	\end{equation}  
	with $\tilde{\omega}_{i} = \omega_{i} (m_{j_{1}} -\lambda_{i,p})(m_{j_{1}+1} -\lambda_{i,p}) > 0$ and 
	\begin{equation}
	\label{express2}
	\int (x-t)^{2p-3} d\mu^{{\tilde{\Omega}}}_{R_{m,2p-1}} = 
	\sum_{i=1, i\not= j_{1},j_{1}+1 }^{2p-1} \tilde{\alpha}_{i} (x-m_{i})^{2p-3}, 
	\end{equation}  
	with $\tilde{\alpha_i} = (m_{j_{1}} - m_{i})(m_{j_{1}+1} - m_{i}) > 0$. Applying the same arguments 
	as above to (\ref{express1}) and (\ref{express2}). Sylvester's theorem, condition (\ref{introm}) and Theorem \ref{generalquadrature} applied to the measure $d\mu^{{\tilde{\Omega}}}_{R_{m,2p-1}}$, enable us to conclude the existence of an odd index $j_{2}\leq 2p-1$ ($j_{2} \not= j_{1}$) 
	and an integer $k_{2}\leq p$ ($k_{2} \not= k_{1}$) such that $m_{j_{2}} \leq \lambda_{k_{2},p} \leq m_{j_{2}+1} =  m_{j_{2}}+1$. 
	If $\lambda_{k_{2},p} < m_{j_{2}+1}$ then the proposition is proved. If $\lambda_{k_{2},p} =m_{j_{2}+1}$ then we 
	iterate the above process by taking the Christoffel transform of $\mu_{R_{m,2p-1}}$ 
	with respect to the annihilator polynomial $(t-m_{j_{1}})(t - m_{j_{1}+1})(t-m_{j_{2}})(t - m_{j_{2}+1})$ 
	and so on. In the course of this iterative process we either find an odd index $j_{s}\leq 2p-1$ and an integer $k_{s}\leq p$ such that 
	$m_{j_{s}} \leq \lambda_{k_{s},p} < m_{j_{s}+1} =  m_{j_{s}}+1$ and in this case the proposition is proved, or we find  
	that all the zeros $\lambda_{1,p} <\lambda_{2,p} <\ldots<\lambda_{p,p}$ of the Charlier-Poisson polynomial $C(.,R_{m,2p-1})$ are integers
	with $m_{2k-1} = \lambda_{k,p}$ for $k=1,2,\ldots,p$ and the coefficients associated with $m_{2k}, k=1,2,\ldots,p-2$ are equal 
	to zero. In this case we write the integer sequence $(m_1,m_2,\ldots,m_{2p-1})$ in a form that answers the 
	claim of the proposition,\ie 
	\begin{equation}
	\label{IntegerRepresentation}
	(\lambda_{1,p},\lambda_{1,p}+1,\ldots,\lambda_{p-1,p},\lambda_{p-1,p}+1,\lambda_{p,p}).
	\end{equation}    
	That the integers in the representation (\ref{IntegerRepresentation}) are pairwise distinct is a consequence 
	of the well-known fact that there is at least one integer between two consecutive zeros of 
	discrete orthogonal polynomials \cite{karlin}. 
\qed         
\end{proof}

One can iterate Proposition \ref{SingleIteration} as follows. We know 
from Proposition \ref{SingleIteration} that there exist an odd index $j_{1}\leq 2p-1$ and an integer 
$k_{1}\leq p$ such that $m_{j_{1}} = \floor{\lambda_{k_{1},p}}$ and $m_{j_{1}+1} = \floor{\lambda_{k_{1},p}}+1$. 
Define the annihilator polynomial $\Omega_{1}(t) = (m_{j_{1}} -t) (m_{j_{1}+1} -t)$. 
By Gauss quadrature with respect to the Christoffel transform measure $\mu_{R_{m,2p-1}}^{\Omega_{1}}$, we have 
\begin{equation*}
\int (x-t)^{2p-3} d\mu^{\Omega_{1}}_{R_{m,2p-1}} = \sum_{i=1}^{p-1} \omega^{1}_{i} (x-\lambda^{\Omega_{1}}_{i,p-1})^{2p-3}, 
\quad \omega^{1}_i > 0 \quad \textnormal{for} \quad i=1,2,\ldots,p-1.
\end{equation*}   
and evaluating the polar form to both sides of (\ref{representation2}) at $(m_{j_{1}},m_{j_{1}+1},x^{[2p-3]})$ yields 
\begin{equation*}
\int (x-t)^{2p-3} d\mu^{\Omega_{1}}_{R_{m,2p-1}} = 
\sum_{i=1, i\not= j_1,j_1+1}^{2p-1} \alpha^{1}_{i} (x-m_{i})^{2p-3}, \;
\end{equation*}
with $\alpha^{1}_{i}= \alpha_i (m_{j_{1}}- m_{i})(m_{j_{1}+1} - m_{i}) > 0$.
Thus using the same arguments as in the proof of Proposition \ref{SingleIteration}, 
we conclude that there exist an odd integer $j_2 \leq 2p-3$ ($j_2 \not=j_1,j_{1}+1$)
and an integer $k_2 \leq p-1$ such that 
\begin{equation*}
m_{j_{2}} = \floor{\lambda^{\Omega_{1}}_{k_{2},p-1}} \quad \textnormal{and} \quad m_{j_{2}+1} = 
\floor{\lambda^{\Omega_{1}}_{k_{2},p-1}}+1.
\end{equation*} 
We can again iterate the same argument this time with the annihilator polynomial 
$\Omega_{2}(t) = (m_{j_{1}} -t) (m_{j_{1}+1} -t) (m_{j_{2}} -t) (m_{j_{2}+1} -t)$. 
Obviously, the above process terminates after $(p-1)$ iterations and leads to the 
following theorem.

\begin{theorem}
	\label{MultipleIteration}
	Let $\Phi_{m,2p-1}$ the optimal threshold polynomial with optimal threshold factor $R_{m,2p-1}$
	\begin{equation*}
	\Phi_{m,2p-1}(x) = \sum_{k=1}^{2p-1} \alpha_k \left(1 + \frac{x}{R_{m,2p-1}}\right)^{m_k}.
	\end{equation*}
	We can arrange the integer sequence $(m_1,m_2,\ldots,m_{2p-3},m_{2p-2},m)$ as
	\begin{equation*}
	(m_{j_1},m_{j_{1}+1},m_{j_2},m_{j_{2}+1}\ldots m_{j_p},m_{j_{p}+1},m)
	\end{equation*}
	such that for any $1 \leq \ell \leq p$, there exists an integer $1 \leq k_{\ell} \leq p -\ell +1$
	such that    
	\begin{equation}
	\label{RelativeLocation}
	m_{j_{\ell}} = \floor{\lambda^{\Omega_{\ell-1}}_{k_{\ell},p-\ell+1}} \quad \textnormal{and} \quad 
	m_{j_{\ell}+1} =\floor{\lambda^{\Omega_{\ell-1}}_{k_{\ell},p-\ell+1}}+1,
	\end{equation} 
	where $(\lambda_{1,p-\ell+1}^{\Omega_{\ell-1}},\lambda_{2,p-\ell+1}^{\Omega_{\ell-1}},\ldots,
	\lambda_{p-\ell+1,p-\ell+1}^{\Omega_{\ell-1}})$ are the zeros of the degree $(p-\ell+1)$ 
	orthogonal polynomial $\Pi^{\Omega_{\ell-1}}_{p-\ell+1}$ where the annihilator polynomial 
	$\Omega_{\ell-1}$ is given by $\Omega_{\ell-1}(t) = \prod_{i=1}^{\ell-1} (m_{j_{i}}-t)(m_{j_{i}+1}-t)$
	and $\Omega_{0}(t) \equiv 1$.
\end{theorem} 

\begin{figure*}[h!]
	\hskip -0.7 cm
	\begin{overpic}[width=0.57\textwidth]{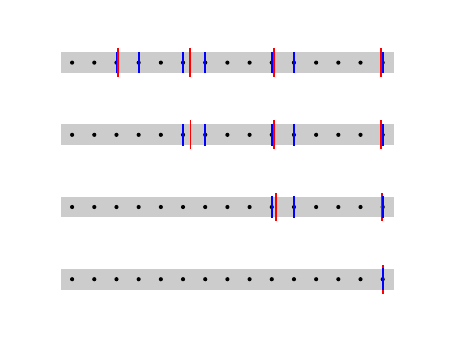}
		\put(10,67){\small{$(a)$}}
		\put(7.3,61){\small{$(1)$}}
		\put(7.3,45){\small{$(2)$}}
		\put(7.3,30){\small{$(3)$}}
		\put(7.3,15){\small{$(4)$}}
	\end{overpic}
	\hskip -1.2 cm
	\begin{overpic}[width=0.57\textwidth]{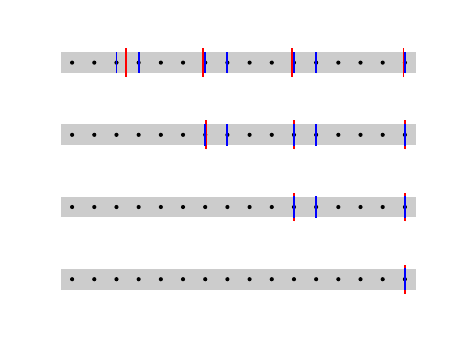}
		\put(10,67){\small{$(b)$}}
	\end{overpic}
	\vskip -1 cm
	\caption{Relative location of the exponents of the optimal threshold polynomial and the zeros of the adpative
		Christoffel transforms of the initial measure. (See Example \ref{ExampleChristoffel} for explanatory details)}
	\label{fig:Christoffel}	
\end{figure*}

Let us illustrate Theorem \ref{MultipleIteration} by an example.
\begin{example}
	\label{ExampleChristoffel}
	Using the algorithm of Section 8, it can be shown that $R_{14,7} \simeq 6.0907$ 
	is the unique positive real zero of the polynomial equation 
	\begin{equation*}
	R^7 - 28R^6 + 380R^5 - 3260R^4 + 19080R^3 - 75960R^2 + 189360R - 226800 =0.
	\end{equation*}
	Moreover, the exponents of the optimal threshold polynomial $\Phi_{14,7}$ are 
	\begin{equation*}
	(m_1,m_2,m_3,m_4,m_5,m_6,m_{7}) = (2,3,5,6,9,10,14).
	\end{equation*} 
	Figure \ref{fig:Christoffel} ((a);(1)) shows the location of the exponents 
	$m_i, i=1,2,\ldots,7$ (blue bars) relative to the location of the zeros of the 
	Poisson-Charlier polynomial $C_{4}(.,R_{14,7})$ (red bars). According to Theorem 
	\ref{MultipleIteration}, there is a least one zero of $C_{4}(.,R_{14,7})$ that is between exponents 
	of the form $(q,q+1)$ of the optimal threshold polynomial. In this specific case, each zero 
	(except for the largest one) of $C_{4}(.,R_{14,7})$ is between two exponents of the form 
	$(q,q+1)$ of $\Phi_{14,7}$ (see Figure \ref{fig:Christoffel} ((a);(1)).
	Figure \ref{fig:Christoffel} ((a);(2)) shows the location of the zeros of the degree $3$ 
	orthogonal polynomial $\Pi_{3}^{\Omega_{1}}$ of the  Christoffel transform measure $\mu^{\Omega_{1}}_{R_{m,2p-1}}$  
	associated with the annihilator polynomial $\Omega_{1}(t)= (m_{1}-t)(m_{2}-t)$ (red bars) and  
	the location of the remaining exponents $(m_3,m_4,m_5,m_6,m_7)$ of $\Phi_{14,7}$. 
	Here again and in accordance with Theorem \ref{MultipleIteration}, there exists a least one zero 
	of $\Pi_{3}^{\Omega_{1}}$ that is between exponents of the form $(q,q+1)$ of the remaining exponents.
	In this specific example, each zero (except for the largest one) of $\Pi_{3}^{\Omega_{1}}$ 
	is between two exponents of the form $(q,q+1)$ (see Figure \ref{fig:Christoffel}, (a),(2)). 
	Similarly, Figure \ref{fig:Christoffel} ((a);(3)) shows the zeros of $\Pi_{2}^{\Omega_2}$ with 
	annihilator polynomial $\Omega_{2}(t) = (t-m_1)(t-m_2)(t-m_3)(t-m_4)$ relative to the remaining 
	exponents $(m_5,m_6,m_7)$. Finally as                        
	\begin{equation*}
	h_7(m_1,m_2,m_3,m_4,m_5,m_6,m_7;R_{14,7}) =0,
	\end{equation*}
	the zero of $\Pi_{1}^{\Omega_{3}}$ with $\Omega_{3}(t) = \prod_{i=1}^{6}(t-m_i)$ is equal to $m=14$. 
	This is confirmed by Figure \ref{fig:Christoffel}, ((a);(4)). 
	This specific example offers various choice for the exponents where we can perform 
	the Christoffel transformation at each stage of the iterative process. 
	This is in contrast with the case of $R_{15,7}$ and its associated polynomial $\Phi_{15,7}$.  
	Using the algorithm described in Section 8, we find that $R_{15,7} \simeq 6.8035$ 
	is the unique positive real zero of the polynomial equation 
	\begin{equation*}
	R^7 - 33R^6 + 516R^5 - 4956R^4 + 31500R^3 - 132300R^2 + 340200R - 415800 =0.
	\end{equation*}
	The exponents of the optimal threshold polynomial $\Phi_{15,7}$ are 
	\begin{equation*}
	(m_1,m_2,m_3,m_4,m_5,m_6,m_{7}) = (2,3,6,7,10,11,15).
	\end{equation*}          
	As can be seen from Figure 1 (b), at each stage of the iterative process described in Theorem \ref{MultipleIteration} 
	there exists only one zero of the Christoffel transform orthogonal polynomials that is between exponents of the form 
	$(q,q+1)$ of the optimal threshold polynomial.
\end{example} 
\subsection{Integral spectrum and integral spectral radius}
Let $\Omega$ be an admissible polynomial and $\lambda_{1,n}^{\Omega}(R),\ldots,
\lambda_{n,n}^{\Omega}(R)$ the zeros of the orthogonal polynomial $\Pi_{n}^{R,\Omega}$.  

\begin{theorem}
	\label{StrictlyIncreasing}
	The functions $\lambda_{k,n}^{\Omega}(R), k=1,2,\ldots,n,$ are strictly 
	increasing functions with respect to the parameter $R$. 
\end{theorem}
\begin{proof}
	The proof of the theorem follows a classical argument by Markov. 
	By Gauss quadrature, for any polynomial $P$ of degree $2n-1$, we have
	\begin{equation*}
	\int P(t) d\mu_{R}^{\Omega}(t) = \sum_{i=1}^{n} \alpha_i(R) P(\lambda_{i,n}^{\Omega}(R)), 
	\end{equation*} 
	or equivalently, 
	\begin{equation}
	\label{Gquad}
	e^{-R} \sum_{j=0}^{\infty} P(j)\Omega(j) \frac{R^{j}}{j!}
	=\sum_{i=1}^{n} \alpha_i(R) P(\lambda_{i,n}^{\Omega}(R)).
	\end{equation} 
	Differentiating (\ref{Gquad}) with respect to the parameter $R$, we obtain 
	\begin{equation}
	\label{MarkovTrick}
	\begin{split}
	-\int P(t) d\mu_{R}^{\Omega}(t) + e^{-R} \sum_{j=0}^{\infty} P(j+1)\Omega(j+1) \frac{R^{j}}{j!} & =  \\
	\sum_{i=1}^{n} \frac{\partial \alpha_i(R)}{\partial R} P(\lambda_{i,n}^{\Omega}(R)) +  \sum_{i=1}^{n} \alpha_i(R) 
	\frac{\partial \lambda_{i,n}^{\Omega}(R)}{\partial R} P'(\lambda_{i,n}^{\Omega}(R)).
	\end{split}
	\end{equation}
	Now we specialize our analysis to the polynomials 
	\begin{equation*}
	P_{k}(t) = \frac{(\Pi_{n}^{R,\Omega}(t))^2}{t-\lambda_{k,n}^{\Omega}(R)}, \quad k=1,2,\ldots,n.
	\end{equation*}
	Each polynomial $P_{k}$ satisfy $P_{k}(\lambda_{i,n}^{\Omega}(R)) =0$
	for any $1 \leq i \leq n$,   $P_{k}'(\lambda_{i,n}^{\Omega}(R)) = 0$ 
	for any $i\not= k$ and $P_{k}'(\lambda_{k,n}^{\Omega}(R)) > 0.$\footnote{Here, the orthogonal polynomials $\Pi_{n}^{R,\Omega}$ are normalized to be monic} 
	Thus identity (\ref{MarkovTrick}) applied to the polynomial $P_{k}$ leads to 
	\begin{equation}
	\label{Integral+1}
	\int P_{k}(t+1) \Omega(t+1) d\mu_R(t) = \alpha_k(R) 
	\frac{\partial \lambda_{k,n}^{\Omega}(R)}{\partial R} P_{k}'(\lambda_{k,n}^{\Omega}(R)).
	\end{equation}
	Noting that 
	\begin{equation*}
	\int P_{k}(t+1) \Omega(t+1) d\mu_R(t) = \frac{1}{R} \int t P_{k}(t) \Omega(t) d\mu_R(t),
	\end{equation*} 
	we have 
	\begin{equation*}
	\begin{split}
	\int P_{k}(t+1) \Omega(t+1) d\mu_R(t) & = \int P_{k}(t+1) \Omega(t+1) d\mu_R(t) 
	- \hskip -0.05 cm  \frac{\lambda_{k,n}^{\Omega}(R)}{R} \int \hskip -0.1 cm P_{k}(t) \Omega(t) d\mu_R(t) \\
	& = \frac{1}{R}\int (t-\lambda_{k,n}^{\Omega}(R)) P_{k}(t)\Omega(t) d\mu_R(t) \\
	& = \frac{1}{R} \int (\Pi_{n}^{R,\Omega}(t))^2 d\mu_{R}^{\Omega}(t) > 0.
	\end{split}
	\end{equation*}
	Thus, from (\ref{Integral+1}), we deduce that $\frac{\partial \lambda_{k,n}^{\Omega}(R)}{\partial R}>0$. 
	This concludes the proof.
\qed
\end{proof}

\begin{definition}
	A finite sequence of non-negative integers ${\bf{n}} = \{n_1,n_2,\ldots,n_{p-1}\}$ 
	is called a $p$-configuration if for $k=1,\ldots, p-1$
	\begin{equation*}
	1 \leq n_{k} \leq p-k+1. 
	\end{equation*} 
\end{definition}
Given a positive real number $R$ and a $p$-configuration ${\bf{n}} = \{n_1,n_2,\ldots,n_{p-1}\}$,
we construct a sequence of integers $(m_1,m_2,\ldots,m_{2p-3},m_{2p-2})$ iteratively as follows:
\begin{equation}
\label{FirstCouple}
m_1 = \floor{\lambda_{n_1,p}^{\Omega_{0}}(R)}; 
\quad
m_2 = m_1 + 1, 
\end{equation}
where $\lambda_{1,p}^{\Omega_{0}}(R)<\lambda_{2,p}^{\Omega_{1}}(R)<\ldots<\lambda_{p,p}^{\Omega_{0}}(R)$ are 
the zeros of the Poisson-Charlier orthogonal polynomial $C_{p}(.,R)$. 
The integers $m_{2k-1},m_{2k}$ are defined by
\begin{equation}
\label{AllCouple}
m_{2k-1} = \floor{\lambda_{n_{k},p-k+1}^{\Omega_{k-1}}(R)}; 
\quad
m_{2k} = m_{2k-1} + 1, 
\end{equation} 
where the annihilator polynomial  $\Omega_{k-1}$ is given by 
$\Omega_{k-1}(t) = \prod_{i=1}^{k-1} (m_{2i-1}-t)(m_{2i}-t)$

\begin{definition}
	The integer sequence $\mathcal{M}({\bf{n}},R) = \{m_1,m_2,\ldots,m_{2p-3},m_{2p-2}\}$ 
	obtained from (\ref{FirstCouple}) and (\ref{AllCouple}) is termed the integral spectrum 
	associated with the $p$-configuration ${\bf{n}}$ and the positive real number $R$. 
	Moreover, the unique zero $\rho({\bf{n}},R)$ of the degree $1$ polynomial orthogonal 
	with respect to the measure $\mu_{R}^{\Omega_{p-1}}$ is called the spectral radius 
	with respect to the couple $({\bf{n}},R)$. The real number $\rho({\bf{n}},R)$ is given 
	explicity by 
	\begin{equation}
	\label{SpectralDefinition}
	\rho({\bf{n}},R) = \frac{\int t P(t) d\mu_{R}(t)}{\int P(t) d\mu_{R}(t)} 
	\quad \textnormal{where} \quad 
	P(t) = \prod_{j \in \mathcal{M}({\bf{n}},R)} (t - j).
	\end{equation}   
\end{definition}
\begin{figure*}[h!]
	\hskip -0.7 cm
	\begin{overpic}[width=0.57\textwidth]{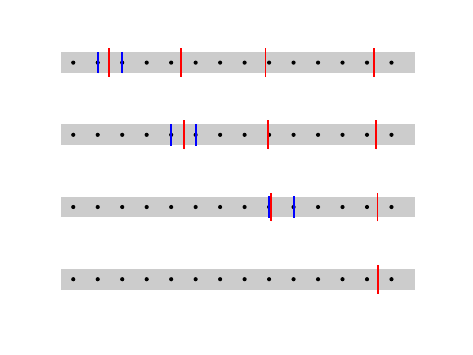}
		\put(10,69){\small{$(a)$}}
	\end{overpic}
	\hskip -1.2 cm
	\begin{overpic}[width=0.57\textwidth]{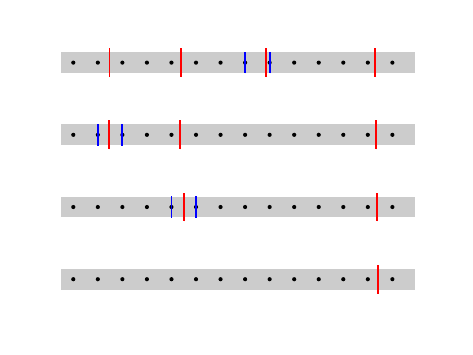}
		\put(10,69){\small{$(b)$}}
	\end{overpic}
	\vskip -1cm
	\caption{The integral spectrum associated with $R=5$ and the configurations  
		${\bf{n_{1}}} = \{1,1,1\}$ (a) and ${\bf{n_{2}}} = \{3,1,1\}$ (b) }
	\label{fig:IntegralSpectrum}	
\end{figure*}
\begin{example}
	Figure \ref{fig:IntegralSpectrum} shows the integral spectrum (bleu bars) associated with 
	$R=5$ and the configurations ${\bf{n_{1}}} = \{1,1,1\}$ (Figure \ref{fig:IntegralSpectrum}; a)
	and ${\bf{n_{2}}} = \{3,1,1\}$ (Figure \ref{fig:IntegralSpectrum}; b). The red bars show the zeros of the 
	orthogonal polynomials $\Pi_{4-k+1}^{R,\Omega_{k-1}}$ for $k=1,2,3,4$ respectively. In this example,
	we find  
	\begin{equation*}
	\begin{split}
	& \mathcal{M}({\bf{n_{1}}},R) = \{1,2,4,5,8,9\} 
	\quad \textnormal{and} \quad \rho({\bf{n_{1}}},R) \simeq 12.4539, \\
	& \mathcal{M}({\bf{n_{2}}},R) = \{7,8,1,2,4,5\} 
	\quad \textnormal{and} \quad \rho({\bf{n_{1}}},R) \simeq 12.4134.
	\end{split}
	\end{equation*} 
\end{example}

\begin{proposition}
	\label{SpectralMonotone}
	Given a $p$-configuration ${\bf{n}}$, the spectral radius $\rho({\bf{n}},R)$
	is a continuous and strictly increasing function of the parameter $R$.
\end{proposition}   
\begin{proof}
	We first prove that the function $\rho({\bf{n}},R)$ is a right-continuous
	function of the parameter $R$. Let $(m_1,m_2,\ldots,m_{2p-3},m_{2p-2}) = 
	\mathcal{M}({\bf{n}},R)$. By definition, the following inequalities hold 
	\begin{equation}
	\label{LocalInequalities}
	m_{2k-1} \leq \lambda_{n_{k},p-k+1}^{\Omega_{k-1}}(R) < m_{2k}, \quad k=1,2,\ldots,p-1.
	\end{equation} 
	Let $\epsilon >0$ be a small enough positive real number. Thus an iterative use of 
	Theorem \ref{StrictlyIncreasing} shows that 
	\begin{equation*}
	m_{2k-1} \leq \lambda_{n_{k},p-k+1}^{\Omega_{k-1}}(R+\epsilon) < m_{2k}, \quad k=1,2,\ldots,p-1.
	\end{equation*} 
	Therefore, $\mathcal{M}({\bf{n}},R) = \mathcal{M}({\bf{n}},R+\epsilon)$ and the right-continuity
	of $\rho({\bf{n}},R)$ follows from the fact that for any $r$ such that $R \leq r \leq R+\epsilon$, 
	$\rho({\bf{n}},r)$ is given by the continuous function in $r$ 
	\begin{equation*}
	\rho({\bf{n}},r) = \frac{\int t P(t) d\mu_{r}(t)}{\int P(t) d\mu_{r}(t)} 
	\quad \textnormal{where} \quad 
	P(t) = \prod_{j \in \mathcal{M}({\bf{n}},R)} (t - j).
	\end{equation*}   
	Proving the left-continuity of the function $\rho({\bf{n}},R)$ is technically more difficult.
	The main reason is that some of the left inequalities (\ref{LocalInequalities}) can become equalities. 
	Suppose there exists a certain $k \leq p-1$ such that $m_{2k-1} = \lambda_{n_{k},p-k+1}^{\Omega_{k-1}}(R)$.
	Then, for any positive real number $\epsilon$, no matter how small it is, 
	we have $\mathcal{M}({\bf{n}},R) \not= \mathcal{M}({\bf{n}},R-\epsilon)$ and
	thus we cannot use the same arguments as in the proof of the right-continuity. Therefore, 
	if all the inequalities in (\ref{LocalInequalities}) as strict one for a given $R$ 
	then we can use exactly the same arguments as in the proof of the right-continuity to show the 
	left-continuity of the function $\rho({\bf{n}},R)$ at $R$. Let us now take an $R$ for which  
	there exist integers $k\leq p-1$ with $m_{2k-1} = \lambda_{n_{k},p-k+1}^{\Omega_{k-1}}(R)$ where 
	$(m_1,m_2,\ldots,m_{2p-3},m_{2p-2}) = \mathcal{M}({\bf{n}},R)$. Denote by $\mathcal{I}$ the set of 
	indices $k$ such that  $m_{2k-1} = \lambda_{n_{k},p-k+1}^{\Omega_{k-1}}(R)$.
	Define the real numbers $\tilde{\mathcal{M}}=(\tilde{m}_1,\tilde{m}_2,\ldots,\tilde{m}_{2p-2})$ by 
	\begin{equation*}
	\begin{split}
	\tilde{m}_{2i-1} = m_{2i-1} \quad &\textnormal{for} \quad i=1,2,\ldots,p-1, \\ 
	\tilde{m}_{2i} = m_{2i} \quad &\textnormal{if} \quad 2i-1 \notin \mathcal{I}, \\
	\tilde{m}_{2i} = m_{2i-1} - \frac{1}{2} \quad &\textnormal{if} \quad 2i-1 \in \mathcal{I}.
	\end{split}   
	\end{equation*} 
	Obviously, we have 
	\begin{equation}
	\label{NewInequalities}
	\tilde{m}_{2k-1} <\lambda_{n_{k},p-k+1}^{\Omega_{k-1}}(R) \leq \tilde{m}_{2k} 
	\quad \textnormal{for} \quad k=1,2,\ldots,p-1.
	\end{equation}
	We shall show that 
	\begin{equation}
	\label{ImportantRho}
	\rho({\bf{n}},R) = \frac{\int t Q(t) d\mu_{R}(t)}{\int Q(t) d\mu_{R}(t)} 
	\quad \textnormal{with} \quad 
	Q(t) = \prod_{k=1}^{2p-2} (t - \tilde{m}_k).
	\end{equation}
	Once (\ref{ImportantRho}) shown, the proof of the left-continuity follows the same arguments as the one 
	we used for the right-continuity. Namely, we take a small enough positive real number $\epsilon$ and by
	an iterative application of Theorem \ref{StrictlyIncreasing} while taking into account 
	inequalities (\ref{NewInequalities}), leads to 
	\begin{equation*}
	\tilde{m}_{2k-1} <\lambda_{n_{k},p-k+1}^{\Omega_{k-1}}(R-\epsilon) \leq \tilde{m}_{2k}, 
	\quad k=1,2,\ldots,p-1.
	\end{equation*}
	Therefore, the left-continuity becomes a consequence of the fact that for any real number $r$ such that 
	$R-\epsilon \leq r \leq R$,  $\rho({\bf{n}},r)$ is given by the continuous function in $r$ 
	\begin{equation*}
	\rho({\bf{n}},r) = \frac{\int t Q(t) d\mu_{r}(t)}{\int Q(t) d\mu_{r}(t)} 
	\quad \textnormal{with} \quad 
	Q(t) = \prod_{k=1}^{2p-2} (t - \tilde{m}_k).
	\end{equation*}    
	To show (\ref{ImportantRho}) we proceed as follows. Let $k$ be the smallest index in $\mathcal{I}$. 
	By Gauss quadrature with respect to the measure $\mu^{\Omega_{k-1}}_{R}$ , we have
	\begin{equation}
	\label{GaussLocal}
	\int (x-t)^{2(p-k)+1} d\mu^{\Omega_{k-1}}_{R} = \sum_{j=1}^{p-k+1} \alpha_j 
	(x-\lambda^{\Omega_{k-1}}_{j,p-k+1})^{2(p-k)+1},
	\end{equation}
	with $\lambda_{n_{k},p-k+1}^{\Omega_{k-1}} = m_{2k-1}$ and $\alpha_j > 0$ for $j=1,2,\ldots,p-k+1$.   
	Let $\eta$ be a real number such that $m_{2k-1}-1 \leq \eta \leq m_{2k-1} +1$.  
	Evaluating the polar form of both sides of (\ref{GaussLocal}) at $(x^{[2(p-k)-1]},m_{2k-1},\eta)$, 
	we obtain
	\begin{equation}
	\label{NewGauss}
	\int (x-t)^{2(p-k)-1} (m_{2k-1}-t)(\eta-t)d\mu^{\Omega_{k-1}}_{R} = \sum_{j=1, j\not={n}_k}^{p-k+1} 
	\tilde{\alpha}_j (x-\lambda^{\Omega_{k-1}}_{j,p-k+1})^{2(p-k)-1}, 
	\end{equation}  
	with $\tilde{\alpha}_j = \alpha_{j} (m_{2k-1} -\lambda^{\Omega_{k-1}}_{j,p-k+1}) (\eta -\lambda^{\Omega_{k-1}}_{j,p-k+1}) > 0$ 
	for $j=1,\ldots,p-k+1; j \not= n_{k}$. This positivity is a direct consequence of the fact that the orthogonal polynomial 
	$\Pi^{\Omega_{k-1}}_{p-k-1}$ has no zero other than $m_{2k-1} = \lambda^{\Omega_{k-1}}_{n_k,p-k+1}$ in the interval 
	$[m_{2k-1}-1,m_{2k-1}+1]$. Identity (\ref{NewGauss}) is thus the Gauss quadrature with respect to the 
	Christoffel transform measure with annihilator polynomial $\tilde{\Omega}_{\eta}(t) = (m_{2k-1}-t)(\eta-t) \Omega_{k-1}(t)$. 
	This shows in particular that $\lambda^{\Omega_{k-1}}_{j,p-k+1}, j=1,2,\ldots,p-k+1; j \not= n_k,$ 
	are the zeros of the orthogonal polynomial $\Pi^{\tilde{\Omega}_{\eta},R}_{p-k}$ and that these zeros 
	does not depend on the real number $\eta$. From the definition of the integral spectrum  
	$\mathcal{M}({\bf{n}},R)$, if we take $\eta = m_{2k}$, then up to an adequate normalization, we have 
	\begin{equation}
	\label{OrthogonalEquality}
	\Pi^{\tilde{\Omega}_{\eta},R}_{j} = \Pi^{\Omega_{k},R}_{j}, \quad j=1,2,\ldots
	\end{equation}
	If instead of $\eta =  m_{2k}$, we take  $\eta =  m_{2k-1} - 1/2$, then we still have (\ref{OrthogonalEquality}) 
	up to an adequate normalization. Accordingly, changing the value of $m_{2k}$ to $\tilde{m}_{2k} = m_{2k-1} - 1/2$ 
	in the integral spectrum $\mathcal{M}({\bf{n}},R)$ does not change the value of the spectral radius $\rho({\bf{n}},R)$  
	in the sense that we have 
	\begin{equation*}
	\rho({\bf{n}},R) =  
	\frac{\int t Q_1(t) d\mu_{R}(t)}{\int Q_1(t) d\mu_{R}(t)} = 
	\frac{\int t Q_2(t) d\mu_{R}(t)}{\int Q_2(t) d\mu_{R}(t)},  
	\end{equation*}
	where $Q_1(t) = \prod_{j=1}^{2p-2}(t - m_{j})$ and 
	$Q_2(t) = (t-\tilde{m}_{2k}) \prod_{j=1, j\not=2k}^{2p-2}(t - m_{k})$. 
	Applying iteratively the same arguments to each index in $\mathcal{I}$ (from the smallest index to the largest one),
	proves (\ref{ImportantRho}). To prove that the function $\rho({\bf{n}},R)$ is a strictly increasing function 
	of the parameter $R$ it is sufficient to prove that for any $\epsilon >0$ small 
	enough, we have $\rho({\bf{n}},R+\epsilon) > \rho({\bf{n}},R)$. As already 
	shown, we can always choose an $\epsilon >0$ small enough such that  
	\begin{equation*}
	\mathcal{M}({\bf{n}},R+\epsilon) =  \mathcal{M}({\bf{n}},R).
	\end{equation*}
	According to Theorem \ref{StrictlyIncreasing},  $\rho({\bf{n}},R+\epsilon) > \rho({\bf{n}},R)$ since 
	for any $R\leq r\leq R+\epsilon$, $\rho({\bf{n}},r)$ is the unique zero of the orthogonal polynomial 
	$\Pi_{1}^{\Omega_{p-1},R}$ where $\Omega_{p-1}(t) = \prod_{j=1}^{2p-2} (t - m_{j})$. 
\qed
\end{proof}

Given a positive integer $m$ and a $p$-configuration ${\bf{n}}$, with $m \geq 2p-1$, then according to Proposition 
\ref{SpectralMonotone}, there exists a unique real number $\mathcal{R}_{m,p}({\bf{n}})$ 
such that
\begin{equation*}
\rho\left( {\bf{n}},\mathcal{R}_{m,p}({\bf{n}})\right) = m.
\end{equation*}
The real number $\mathcal{R}_{m,p}({\bf{n}})$ will be called the {\it{optimal spectral radius}}
with respect to the $p$-configuration ${\bf{n}}$ and the integer $m$. The associated integral spectrum 
\begin{equation*}
\mathcal{M}({\bf{n}},\mathcal{R}_{m,p}({\bf{n}})) = \{m_1,m_2,\ldots,m_{2p-3},m_{2p-2}\}
\end{equation*}
will be termed the {\it {optimal integral spectrum}} and denoted by $\mathcal{M}_{m,p}({\bf{n}})$.

We shall need the following simple yet, important result. 

\begin{proposition}
	Let $m$ and $p$ be positive integers ($m \geq 2p-1$) and ${\bf{n}}$ be a $p$-configuration.
	If $R_1$ and $R_2$ are two real numbers such that 
	\begin{equation}
	\label{DichotomyInequalities}
	\mathcal{M}({\bf{n}},R_1) = \mathcal{M}({\bf{n}},R_2) 
	\quad and \quad 
	\rho({\bf{n}},R_1) \leq m \leq  \rho({\bf{n}},R_2),
	\end{equation} 
	then the optimal integral spectrum with respect to ${\bf{n}}$ and to the integer $m$ is given by 
	$\mathcal{M}_{m,p}({\bf{n}}) = \mathcal{M}({\bf{n}},R_1)=\mathcal{M}({\bf{n}},R_1)$ 
	and the optimal spectral radius $\mathcal{R}_{m,p}({\bf{n}})$ is the unique positive real 
	zero in $[R_1,R_2]$ of the polynomial equation in $R$
	\begin{equation}
	\label{SpectralRadius}
	h_{2p-1}(m_1,m_2,\ldots,m_{2p-2},m; R) = 0, 
	\end{equation}
	where $(m_1,m_2,\ldots,m_{2p-2}) = \mathcal{M}_{m,p}({\bf{n}})$ and $h_{2p-1}(-,R)$ the polar form of 
	the polynomial $\mathcal{H}_{2p-1}(.;R)$. 
\end{proposition}
\begin{proof}
	As shown in Proposition \ref{SpectralMonotone}, the function  
	$\rho({\bf{n}},R)$ is continuous and strictly increasing with respect to $R$. Thus, from the 
	inequalities (\ref{DichotomyInequalities}), we conclude that there exists a real 
	number $R$ such that  $R_1 \leq R \leq R_2$ and $\rho({\bf{n}},R) = m$. 
	Moreover, we have
	\begin{equation*}
	m_{2k-1} \leq \lambda_{n_{k},p-k+1}^{\Omega_{k-1}}(R_j) < m_{2k}, 
	\quad k=1,2,\ldots,p-1 \quad \textnormal{and} \quad  j=1,2.
	\end{equation*}  
	Thus an iterative use of Theorem \ref{StrictlyIncreasing} shows that for 
	any real number $R$ such that $R_1 \leq R \leq R_2$ we have  
	$\mathcal{M}({\bf{n}},R) = \mathcal{M}({\bf{n}},R_1)$. That the optimal spectral radius 
	$\mathcal{R}_{m,p}({\bf{n}})$ is the unique zero in  $[R_1,R_2]$ of the polynomial equation 
	(\ref{SpectralRadius}) is a direct consequence of (\ref{zerocondition}). 
	This concludes the proof.
\qed
\end{proof}

With the definitions and results shown above, Theorem \ref{MultipleIteration} 
can be restated as follows.

\begin{theorem}
	\label{Configuration1}
	For positive integers $m$ and $p$ such that  $m \geq 2p-1$, 
	there exists a least one $p$-configuration ${\bf{n}}$ such that
	\begin{equation*}
	R_{m,2p-1} = \mathcal{R}_{m,p}({\bf{n}}),
	\end{equation*}  
	and such that the optimal threshold polynomial is given by 
	\begin{equation*}
	\Phi_{m,2p-1}(x) = \sum_{k=1}^{2p-1} \alpha_k \left(1 + \frac{x}{R_{m,2p-1}}\right)^{m_k},
	\end{equation*}
	where $(m_1,m_2,\ldots,m_{2p-3},m_{2p-2}) = \mathcal{M}_{m,p}({\bf{n}})$ 
	and $m_{2p-1} = m$. 
\end{theorem}

The following theorem is an essential part in the algorithmic aspect for computing 
the optimal threshold factors. It roughly states that we do not need to check all 
the $p$-configurations to compute $R_{m,2p-1}$ and its associated optimal 
threshold polynomial. 
\begin{theorem}
	\label{SingleConfiguration}
	Let $m$ and $p$ be positive integers such that $m \geq 2p-1$. 
	Let ${\bf{n}}$ be a $p$-configuration such that the system 
	\begin{equation}
	\label{KindLinearSystem}
	\int (x-t)^{2p-1} d\mu_{\mathcal{R}_{m,p}({\bf{n}})} = 
	\sum_{i=1}^{2p-1} \alpha_{i} (x-m_{i})^{2p-1},
	\end{equation}
	where $(m_1,\ldots,m_{2p-3},m_{2p-2}) = \mathcal{M}_{m,p}({\bf{n}})$
	and $m_{2p-1} = m$, admits a non-negative solution in $(\alpha_1,\alpha_2,\ldots,\alpha_{2p-1})$. Then 
	\begin{equation*}
	R_{m,2p-1} = \mathcal{R}_{m,p}({\bf{n}}) 
	\quad \textnormal{and} \quad 
	\Phi_{m,2p-1}(x) = \sum_{k=1}^{2p-1} \alpha_k \left(1 + \frac{x}{R_{m,2p-1}}\right)^{m_k}.
	\end{equation*}
\end{theorem}
\begin{proof}
	That $\mathcal{R}_{m,p}({\bf{n}}) \leq R_{m,2p-1}$ is a direct consequence of Corollary \ref{HR}.
	For the sake of simplicity, in the rest of the proof, we denote $\mathcal{R}_{m,p}({\bf{n}})$ simply by $R$. 
	Let $\epsilon>0$ be a small enough real number such that
	\begin{equation*}
	\mathcal{M}({\bf{n}},R+\epsilon)= \{m_1,m_2,\ldots,m_{2p-3},m_{2p-2}\},
	\end{equation*}
	and let $\rho({\bf{n}},R + \epsilon)$ be the associated spectral radius. 
	Then according to Proposition \ref{SpectralMonotone}, $\rho({\bf{n}},R + \epsilon) > m$,
	in other word 
	\begin{equation}
	\label{mbound}
	\frac{\int t \prod_{i=1}^{2p-2} (t-m_i) d\mu_{R+\epsilon}}{\int \prod_{i=1}^{2p-2} (t-m_i)
		d\mu_{R+\epsilon}} > m.
	\end{equation}
	The polynomial 
	\begin{equation*}
	f(t) = (m-t) \prod_{k=1}^{2p-2} (t - m_k)
	\end{equation*}  
	satisfies $f(j) \geq 0$ for $j=0,1,\ldots,m$. Moreover, using (\ref{mbound}) we obtain  
	\begin{equation*}
	\int f(t) d\mu_{R + \epsilon} = m \int \prod_{k=1}^{2p-2} (t-m_k)
	d\mu_{R+\epsilon} - \int \prod_{k=1}^{2p-2}  t (t-m_k)
	d\mu_{R+\epsilon} <0.
	\end{equation*} 
	Thus according to Theorem  \ref{FarkasLemma}, $R_{m,2p-1} \leq  R + \epsilon$ 
	for any $\epsilon$ small enough. This completes the proof.
\qed  
\end{proof}

\begin{example}
	\label{AlgorithmExample}
	To illustrate the importance of Theorem \ref{SingleConfiguration}, in this example 
	we compute the value $R_{100,5}$ using the $3$-configuration ${\bf{n}} = \{1,1\}$. 
	From the bounds (\ref{UpperBound}) and (\ref{LowerBound}), we have 
	\begin{equation*}
	\ell_{3}^{(95)} = 81.1972 \leq R_{100,5} \leq  83.023 = \ell_{3}^{(97)}. 
	\end{equation*}  
	Computing the associated integral spectrum and spectral radius, we obtain 
	\begin{equation*}
	\mathcal{M}({\bf{n}},\ell_{3}^{(95)}) = (66, 67, 81, 82) 
	\quad \textnormal{and} \quad 
	\rho({\bf{n}},\ell_{3}^{(95)}) =  98.01. 
	\end{equation*}
	and 
	\begin{equation*}
	\mathcal{M}({\bf{n}},\ell_{3}^{(97)}) = ( 68,69,83,84) 
	\quad \textnormal{and} \quad 
	\rho({\bf{n}},\ell_{3}^{(97)}) =    100.02. 
	\end{equation*}
	As the integral spectrum associated with $\ell_{3}^{(95)}$ and $\ell_{3}^{(97)}$ are different,
	we cannot yet compute the optimal integral spectrum associated with the configuration ${\bf{n}}$
	and the integer $m =100$. Now, one can show that the integral spectrum associated with $R = 82.7$ 
	is given by 
	\begin{equation*}
	\mathcal{M}({\bf{n}}, 82.7) = ( 68,69,83,84) 
	\quad \textnormal{and} \quad 
	\rho({\bf{n}},82.7) =     99.66. 
	\end{equation*}
	Since $\mathcal{M}({\bf{n}}, 82.7) = \mathcal{M}({\bf{n}},\ell_{3}^{(97)})$ and 
	$\rho({\bf{n}},82.7) \leq 100 \leq \rho({\bf{n}},\ell_{3}^{(97)})$ then, according 
	to Proposition \ref{SpectralMonotone}, the optimal integral spectrum associated with the 
	$3$-configuration ${\bf{n}}$ is $(m_1,m_2,m_3,m_4) = (68,69,83,84)$
	and the optimal spectrum radius $\rho({\bf{n}})$ is given by the unique positive zero 
	in $[82.7,83.032]$ of the equation 
	\begin{equation*}
	h_{5}(68,69,83,84,100;R) = 0, \quad i.e; \quad  \rho({\bf{n}}) \simeq  83.002. 
	\end{equation*}  
	Moreover, the coefficients $\alpha_1,\alpha_2,\ldots,\alpha_5$, solution to the 
	associated linear system (\ref{KindLinearSystem}) are given by 
	\begin{equation}
	\label{alphanumbers}
	(\alpha_1,\alpha_2,\ldots,\alpha_5) = ( 0.1188,0.0765,0.2095,0.4539,0.1413).
	\end{equation}  
	Since these coefficients are non-negative, we conclude according to Theorem 
	\ref{SingleConfiguration} that 
	\begin{equation*}
	R_{100,5} = \rho({\bf{n}}) \simeq 83.002,
	\end{equation*}
	or more precisely, $R_{100,5}$ is the unique positive zero of the polynomial equation 
	\begin{equation*}
	R^5 - 394R^4 + 62544R^3 - 5001012R^2 + 201456936R - 3271262400 = 0.
	\end{equation*}
	Moreover, the optimal threshold polynomial is given by
	\begin{equation*}
	\Phi_{100,5} (x) =  \sum_{k=1}^{5} \alpha_k \left(1 + \frac{x}{R_{100,5}}\right)^{m_k},
	\end{equation*}
	where $(m_1,m_2,m_3,m_4,m_5) = (68,69,83,84,100)$ and $\alpha_k, k=1,\ldots,5$, are given by
	(\ref{alphanumbers}). Note that if we want to compute $R_{101,5}$, it is better to start with the bounds 
	\begin{equation*}
	83.002 \simeq R_{100,5} < R_{101,5} \leq \ell_{3}^{(98)} \simeq 83.936.
	\end{equation*} 
\end{example}

\section{Algorithm for the computation of the optimal threshold factors}
Example \ref{AlgorithmExample} exhibits all the ingredients for computing the optimal threshold
factor $R_{m,2p-1}$ and its associated polynomial $\Phi_{m,2p-1}$. 
In this section, we go into the details of the algorithm and show how to improve several of its key ingredients.

\medskip
\noindent
{\bf{Computation of the integral spectrum:}}
Given a $p$-configuration ${\bf{n}}$ and a real number $R$, the computation of 
$\mathcal{M}({\bf{n}},R)$ and $\rho({\bf{n}},R)$ amounts to computing the zeros 
of the orthogonal  polynomials associated with the Christoffel transform measure 
$\mu^{\Omega}_{R}$ with an annihilator polynomial $\Omega$ of the form 
$\Omega(t) = \prod_{k=1}^{2s} (t - m_k)$. The orthogonal polynomials $(\Pi^{\Omega}_{n})_{n\geq 1}$
with respect to $\mu^{\Omega}_{R}$ can be constructed by means of the Christoffel formulas \cite{Gallant1}
\begin{equation*}
\Pi^{\Omega}_{n}(t) = \frac{1}{\Omega(t)} \begin{vmatrix}
C_{n}(t,R) & C_{n+1}(t,R)\ldots & C_{n+2s}(t,R) \\
C_{n}(m_1,R) & C_{n+1}(m_1,R)\ldots & C_{n+2s}(m_1,R) \\
\ldots     & \ldots             & \ldots        \\
C_{n}(m_{2s},R) & C_{n+1}(m_{2s},R)\ldots & C_{n+2s}(m_{2s},R) \\
\end{vmatrix}. 
\end{equation*}   
However, a more efficient method for computing $(\Pi^{\Omega}_{n})_{n\geq 1}$ and their zeros 
consists in deriving the three-term recurrence relation for $(\Pi^{\Omega}_{n})_{n\geq 1}$from the one of 
Poisson-Charlier polynomials. Recall that, if a set of orthogonal polynomials satisfies the three-term 
recurrence relation
\begin{equation}
\label{3-terms}
p_{i}(t)= (t - b_i) p_{i-1}(t) - g_{i} p_{i-2}(t),\quad  i = 1,2,\ldots,
\end{equation}  
with $p_{0}=1$ and $p_{-1}=0$, then the zeros of $p_{n}$ are the eigenvalues of the tridiagonal 
matrix $J_{n}$ 
\begin{equation*}
J_{n} =  
\begin{vmatrix}
b_0        & \sqrt{g_1}   &      &      &      &             \\
\sqrt{g_1} & b_{1}        & \sqrt{g_{2}} &    &  &              \\
& \sqrt{g_2} &  b_{2} &     &   &            \\
&             & .     & .           \\
&             & .     & \sqrt{g_{n-1}}           \\
&              & \sqrt{g_{n-1}}& b_{n-1}         \\
\end{vmatrix}. 
\end{equation*}
Poisson-Charlier polynomials satisfy (\ref{3-terms}) with $b_{i} = R+i-1$ and $g_{i} = (i-1)R$.  
The orthogonal polynomials $(\Pi^{\Omega}_{n})_{n\geq 1}$ satisfy the three-term recurrence relation
\begin{equation*}
\Pi^{\Omega,R}_{i}(t)= (t - B^{(2s)}_i) \Pi^{\Omega,R}_{i-1}(t) - G^{(2s)}_{i} \Pi^{\Omega,R}_{i-2}(t),\quad  i = 1,2,\ldots,
\end{equation*} 
with $\Pi^{\Omega,R}_{0}\equiv 1$ and $\Pi^{\Omega,R}_{-1}\equiv 0$ and the coefficients $B^{(2s)}_{i}$ and $G^{(2s)}_{i}$ are computed
using Algorithm \ref{alg:Christoffel} \cite{Gallant1}
\begin{algorithm}
	\caption{Computing the coefficients of the three-term recurrence relation of the orthogonal polynomials
	 $(\Pi^{\Omega}_{n})_{n\geq 1}$}
	\label{alg:Christoffel}
	-----------------------------------------------------------------------------
	\begin{algorithmic}
		\STATE{Define $B^{(0)}_{j} = b_{j} = R+j-1 ,  G^{(0)}_{j} = g_{j} = (j-1)R; \quad  j = 1, 2,\ldots$}
		\STATE{Define $B^{(k)}_{j}; G^{(k)}_{j}; \quad  j = 1, 2,\ldots; \quad k=1,2,\ldots,2s$} by \\ 
		$E_{0} = 0$ \\
		$Q_{j} = B^{(k-1)}_{j} - E_{j-1} - m_{k}$ \\
		$E_{j} =  G^{(k-1)}_{j+1}/Q_{j}$ \\  
		$B^{(k)}_{j} = m_{k} + Q_{j} + E_{j}$ \\
		$G^{(k)}_{j} = Q_{j} E_{j-1}$ \\
		\RETURN $B^{(2s)}_{j}, G^{(2s)}_{j}$
		
	-----------------------------------------------------------------------------	
	\end{algorithmic}
\end{algorithm}

\medskip
\noindent
{\bf{Computation of the optimal integral spectrum:}}
Let $m$ be a positive integer ($m \geq 2p-1$) and ${\bf{n}}$ be a $p$-configuration.
A consequence of Proposition \ref{SpectralMonotone} is that if $R_1$ and $R_2$ are two real 
numbers such that 
\begin{equation*}
\mathcal{M}({\bf{n}},R_1) = \mathcal{M}({\bf{n}},R_2) 
\quad and \quad 
\rho({\bf{n}},R_1) \leq m \leq  \rho({\bf{n}},R_2),
\end{equation*} 
then the optimal integral spectrum is given by $\mathcal{M}_{m,p}({\bf{n}}) = \mathcal{M}({\bf{n}},R_1)$ 
and the optimal spectral radius $\mathcal{R}_{m,p}({\bf{n}})$ is the unique positive real zero in $[R_1,R_2]$ of the polynomial   
equation
\begin{equation*}
h_{2p-1}(m_1,m_2,\ldots,m_{2p-2},m; R) = 0, 
\end{equation*}
where $(m_1,m_2,\ldots,m_{2p-2}) = \mathcal{M}({\bf{n}})$. In our algorithm, a search for the real numbers $R_1, R_2$ 
is performed via a dichotomy starting from the values $\ell_{p}^{(m-2p+1)}$ and $\ell_{p}^{(m-p)}$ 
obtained in our bounds of the optimal threshold factors. This is achieved using Algorithm \ref{alg:optimalspectrum}.

\vskip -0.4 cm
\begin{algorithm}
	\caption{Computation of the optimal integral spectrum}
	\label{alg:optimalspectrum}
	-----------------------------------------------------------------------------
	\begin{algorithmic}
		\STATE{Define  $R_{max} = \ell_{p}^{(m-p)}$ and $R_{min} = \ell_{p}^{(m-2p+1)}$}
		\WHILE{$\mathcal{M}({\bf{n}},R_{max}) \not= \mathcal{M}({\bf{n}},R_{min})$}
		\STATE{  $R = (R_{min} + R_{max})/2$ \\
			Compute $\mathcal{M}({\bf{n}},R)$ and $\rho({\bf{n}},R)$}
		\IF {$\rho({\bf{n}},R) > m$}
		\STATE{Set $R_{max} = R$, $\mathcal{M}({\bf{n}},R_{max}) = \mathcal{M}({\bf{n}},R)$}  
		\ELSE 
		\STATE{Set $R_{min} = R$, $\mathcal{M}({\bf{n}},R_{min}) = \mathcal{M}({\bf{n}},R)$}
		\ENDIF
		\ENDWHILE
		\STATE{Compute $\mathcal{R}_{m,p}({\bf{n}})$ as the unique positive zero in $[R_{min},R_{max}]$ of 
			$h_{2p-1}(\mathcal{M}({\bf{n}},R),m;r)=0$}
		\RETURN $\mathcal{R}_{m,p}({\bf{n}})$ and $\mathcal{M}_{m,p}({\bf{n}}) = \mathcal{M}({\bf{n}},R)$
		
		-----------------------------------------------------------------------------
	\end{algorithmic}
\end{algorithm}

\medskip
\noindent
{\bf{Computation of the coefficients $\alpha_k, k=1,2,\ldots,2p-1$:}}
Once we compute the optimal integral spectrum $\mathcal{M}_{m,p}({\bf{n}}) = (m_1,m_2,\ldots,m_{2p-2})$
and the optimal radius spectrum $\mathcal{R}_{m,p}({\bf{n}})$ associated with a $p$-configuration ${\bf{n}}$ and a 
positive integer $m$ ($m \geq 2p-1$), we need to check the non-negativity of the coefficients $\alpha_k, 
k=1,2,\ldots,2p-1$. These coefficients can be computed using formulas (\ref{positivitycondition}). 
However, a more efficient method for their computation is to take advantage of the Gaussian 
quadrature with respect to  Poisson-Charlier measures. Indeed, we have 
\begin{equation*}
\mathcal{H}_{2p-1}(x;\mathcal{R}_{m,p}({\bf{n}})) = \sum_{k=1}^{p} \omega_k (x- \lambda_k)^{2p-1},
\end{equation*}    
where $\omega_k, \lambda_k$ are the weights and the nodes of the Gaussian quadrature with respect to the measure 
$\mu_{\mathcal{R}_{m,p}({\bf{n}})}$. Thus using Equations (\ref{positivitycondition}), we obtain 
\begin{equation*}
\alpha_k = \sum_{i=1}^{p} \omega_{i} \prod_{j=1,j\not=k}^{2p} \frac{m_{j} - \lambda_{i}}{m_{j} - m_{i}}, 
\quad k=1,2,\ldots,2p-1, 
\end{equation*} 
where $m_{2p-1}=m$ and $m_{2p} = m+1$. An efficient algorithms for computing the weights and nodes of Gaussian 
quadrature from the three-term recurrence relation is the Golub-Welsch algorithm which can be found in \cite{golub}.


\medskip
\noindent
{\bf{Computation of the optimal threshold factors and associated polynomials:}}
The general algorithm for computing $R_{m,2p-1}$ and $\Phi_{m,2p-1}$ is given in Algorithm \ref{alg:optimalthreshold}.
 
\begin{algorithm}[h!]
	\caption{Computation of the optimal threshold factors}
	\label{alg:optimalthreshold}
	-----------------------------------------------------------------------------
	\begin{algorithmic}
		\STATE{Set  $\alpha_{k} = -1, k=1,2,\ldots,2p-1$}
		\WHILE{one of the $\alpha_k, k=1,2,\ldots,2p-1$ is negative}
		\STATE{Select a $p$-configuration ${\bf{n}}$} 
		\STATE{Compute $\mathcal{M}_{m,p}({\bf{n}})$ and $\mathcal{R}_{m,p}({\bf{n}})$}
		\STATE{Compute  $\alpha_{k}; k=1,2,\ldots,2p-1$  }
		\ENDWHILE
		\RETURN $\mathcal{M}_{m,p}({\bf{n}}),\mathcal{R}_{m,p}({\bf{n}}), \alpha_k, k=1,2,\ldots,2p-1$
		
		-----------------------------------------------------------------------------
	\end{algorithmic}
\end{algorithm}
The initial $p$-configuration ${\bf{n}}$ is chosen to be $\{1,1,\ldots,1\}$. This choice is motivated 
by the observation that this configuration always leads to the optimal threshold factors 
$R_{m,5}$ for $m=5,6,\ldots,2000$,\ie 
\begin{equation*}
R_{m,5} = \mathcal{R}_{m,3}(\{1,1\}) 
\quad \textnormal{for} \quad k=5,6,\ldots,2000.
\end{equation*}
The values of $R_{m,n}$ for $n=5,7,9,11$ and $m=5k, k=1,2,\ldots,40$ are shown in 
Table \ref{tab:table2}. An asterisk indicates a value for which the $p$-configuration 
$\{1,1,\ldots,1\}$ fails to provide for the optimal threshold factor,\ie for which 
\begin{equation*}
R_{m,2p-1} \not= \mathcal{R}_{m,p}(\{1,1,\ldots,1\}).
\end{equation*}
The values of $R_{m,3}$ are not given in Table \ref{tab:table2} since an explicit expression for these values is known and 
given in \cite{kraaPoly}. Note that, for each $m \geq 2p-1$, there may exist many $p$-configurations ${\bf{n}}$ such that
\begin{equation*}
R_{m,2p-1} = \mathcal{R}_{m,p}({\bf{n}} ).
\end{equation*}
Algorithm \ref{alg:optimalthreshold} terminates once it reaches any such
configurations. This partially explain the high efficiency of the algorithm. After the 
initialization of ${\bf{n}}$, we change the $p$-configurations randomly. 
Note that the complexity of the algorithm is independent of the degree $m$ 
of the polynomials.To make this more explicit, let us mention that 
the computational burden for computing either $R_{10,5}$ or $R_{10^{30},5}$ 
is exactly the same. This feature is absent in all the existing algorithms in the 
literature. For instance, in our experiment using Matlab in an Intel(R) Core(TM) 3.20 GHz environment, 
the computation of all the values $R_{m,5}, m=6,7,\ldots,2000$ took less than 84 seconds, 
while the computation of all the values $R_{m,7}, m=8,9,\ldots,2000$ took less than 110 seconds. 
\begin{table}[h!]
	\centering
	\label{tab:table2}
	\begin{tabular}{c|cccc} 	\hline
		&         & \hskip 2cm  $2p-1$        &           &                  \\       
		$m$ &  5      & 7             & 9         & 11               \\
		\hline
		5   &  1      &  --           & --        & --               \\
		10  &  4.8308 &  3.3733       & 2         & --               \\
		15  &  8.5757 & 6.8035        & 5.3363    & 4.1000$^{*}$     \\
		20  & 12.5512 & 10.3955       & 8.6207    & 7.1968           \\
		25  & 16.6426 & 14.1458       & 12.1181$^{*}$   & 10.4401    \\
		30  & 20.8355 & 18.0383       & 15.7996   & 13.8617          \\
		35  & 25.0687 & 21.9991       & 19.5069   & 17.3889          \\
		40  & 29.3824 & 26.0713       & 23.3347   & 21.0411          \\
		45  & 33.6959 & 30.1565       & 27.2467   & 24.7358          \\
		50  & 38.0717 & 34.3177       & 31.1936   & 28.5616$^{*}$    \\
		55  & 42.4783 & 38.5138       & 35.2269   & 32.3888          \\
		60  & 46.9045 & 42.7402       & 39.2661   & 36.3101          \\
		65  & 51.3742 & 47.0065$^{*}$ & 43.3863   & 40.2498          \\
		70  & 55.8354 & 51.2895       & 47.4942   & 44.2407          \\
		75  & 60.3433 & 55.6066       & 51.6671   & 48.2524          \\
		80  & 64.8353 & 59.9393       & 55.8486   & 52.3018          \\
		85  & 69.3699& 64.2863        & 60.0554   & 56.4079          \\
		90  & 73.8924 & 68.6808       & 64.3111   & 60.5126$^{*}$    \\
		95  & 78.4477 & 73.0643       & 68.5542   & 64.6361          \\
		100 & 83.0020 & 77.4830       & 72.8405   & 68.7990          \\
		105  & 87.5746 & 81.8947      & 77.1210   & 72.9860$^{*}$    \\
		110  & 92.1624 & 86.3262      & 81.4319   & 77.1740          \\
		115  & 96.7503 & 90.7940      & 85.7766   & 81.3905          \\
		120  & 101.3649 & 95.2447     & 90.0975   & 85.6207          \\
		125  & 105.9591 & 99.7148     & 94.4558   & 89.8748          \\
		130  & 110.5757 & 104.2063    & 98.8407   & 94.1471          \\
		135  & 115.2069 & 108.6953    & 103.2033  & 98.4163          \\
		140  & 119.8350 & 113.1950    & 107.5950  & 102.7054         \\
		145  & 124.4725 & 117.7118    & 111.9993  & 107.0167         \\
		150  & 129.1137 & 122.2256    & 116.4054  & 111.3141         \\
		155  & 133.7623 & 126.7511    & 120.8271  & 115.6493         \\
		160  & 138.4238 & 131.2958    & 125.2615  & 119.9974         \\
		165  & 143.0776 & 135.8267    & 129.7071  & 124.3378         \\
		170  & 147.7428 & 140.3765    & 134.1436  & 128.6815         \\
		175  & 152.4202 & 144.9467    & 138.6046  & 133.0649         \\
		180  & 157.0889 & 149.4953    & 143.0705  & 137.4426         \\
		185  & 161.7686 & 154.0689    & 147.5367  & 141.8242         \\
		190  & 166.4561 & 158.6436    & 152.0123  & 146.2122         \\
		195 & 171.1420 & 163.2216     & 156.5046  & 150.6200         \\
		200  & 175.8348 & 167.7992    & 160.9934  & 155.0274         \\ \hline
	\end{tabular}
	\caption{Value of the optimal threshold factors $R_{m,2p-1}$ computed using Algorithm 3.
		An asterisk indicates a value for which the $p$-configuration 
		$\{1,1,\ldots,1\}$ fails to provide for the optimal threshold factor.
	}
\end{table}
The values of $R_{m,n}$ and their associated optimal polynomials $\Phi_{m,n}$ 
for different values of $m$ and $n$ are shown is Table \ref{tab:table2}. 
 
For a fair comparison of optimal methods with different number of stages, one should consider 
the effective optimal threshold factor $R^{\text{eff}}_{m,n}$ defined by 
\begin{equation*}
R^{\text{eff}}_{m,n} = \frac{R_{m,n}}{m}.
\end{equation*}             
Taking into account that 
\begin{equation*}
\ell^{(m-2p+1)}_{p} \leq R_{m,2p-1} \leq \ell^{(m-p)}_{p} 
\quad \text{and} \quad  
R_{m,2p-1} = R_{m+1,2p},  
\end{equation*}
and using the following asymptotic formula for the smallest zero, $\ell^{(\alpha)}_{n}$, of the generalized Laguerre polynomials \cite{calogero} 
\begin{equation*}
\ell_{n}^{(\alpha)} = \alpha + \sqrt{2 \alpha} h_1 + \frac{1}{3}(1 + 2n + 2h_1^2) + O\left(\frac{1}{\sqrt{\alpha}}\right) 
\quad \text{as} \quad \alpha \longrightarrow \infty,
\end{equation*}
where $h_1$ is the smallest zero of the Hermite polynomial $H_n$, 
we readily deduce that, for any non-negative integers $m \geq n$, we have 
\begin{equation*}
\lim_{m \to \infty} R^{\text{eff}}_{m,n} = 1.
\end{equation*}
Figure \ref{fig:effective} shows the values of $R^{\text{eff}}_{m,n}$ for different values $n$ and for $n+1 \leq m \leq 200$.
\begin{figure*}[h!]
	\hskip 2 cm
	\begin{overpic}[width=0.7\textwidth]{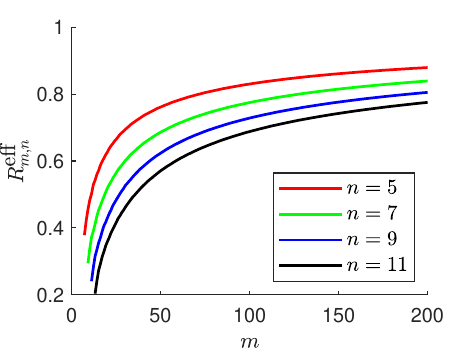}
	\end{overpic}
	\caption{The effective optimal threshold factors $R^{\text{eff}}_{m,n}$ for different values of $m$ and $n$}. 
	\label{fig:effective}	
\end{figure*}

\bigskip

In the following, we list some of the optimal threshold factors $R_{m,n}$ and their optimal threshold polynomials 
$\Phi_{m,n}$ in the form  
\begin{equation*}
\Phi_{m,n}(x) = \sum_{i=1}^{n} \alpha_i \left(1+\frac{x}{R_{m,n}}\right)^{m_i}, 
\end{equation*} 
where we denote by $m = (m_1,m_2,\ldots,m_n)$ and $\alpha = (\alpha_1,\alpha_2,\ldots,\alpha_m)$.

\bigskip

\noindent
$\bullet$ $R_{20,5} \simeq 12.5512$ is the unique positive zero of the polynomial 
\begin{equation*}
R^5 - 52\; R^4 + 1136\; R^3 - 13072\; R^2 + 79352\; R - 203840,
\end{equation*}
and 
\begin{equation*}
m = (7,8,13,14,20), \quad \alpha=(0.1000,0.1528,0.4632,0.1788,0.1052).
\end{equation*}

\bigskip

\noindent
$\bullet$ $R_{73,11} \simeq 46.6584$ is the unique positive zero of the polynomial 
\begin{equation*}
\begin{split}
& R^{11} - 461\; R^{10} + 97862\; R^9 - 12627810\; R^8 + 1100550780\; R^7 - 68023112220\; R^6 + \\
& 3042614045520\; R^5 - 98488786097520\; R^4 + 2261043160437600\; R^3 - \\ 
& 35061624289692000\; R^2 + 330528959503142400\; R - 1435096577615500800
\end{split}
\end{equation*}
and 
\begin{equation*}
m = (27,28,35,36,44,45,52,53,61,62,73), 
\end{equation*}
and 
\begin{equation*}
\alpha=(0.0035,0.0030,0.0073,0.1360,0.3892,0.0739,0.1202,0.2141,0.0029,0.0488).	
\end{equation*}

\noindent
$\bullet$ $R_{500,7} \simeq 448.6201$ is the unique positive zero of the polynomial 
\begin{equation*}
\begin{split}
&R^7 - 3080\; R^6 + 4070910\; R^5 - 2993154240\; R^4 + 1322177180580\; R^3 - \\
& 350893668171600\; R^2 + 51804236904804000\; R - 3282138666126179840
\end{split}
\end{equation*}
and 
\begin{equation*}
m = (401,402,433,434,465,466,500), 
\end{equation*}
and
\begin{equation*}
\alpha=(0.0487,0.0035,0.0079,0.4641,0.2489,0.1866,0.0403).
\end{equation*}

\section{Conclusion}
In this work we provide sharp upper and lower bounds for the optimal threshold 
factors of one-step methods. An efficient algorithm based on adaptive Christoffel 
transformations of Poisson-Charlier measure is proposed. A deep understanding 
of the set of $p$-configurations that lead to the optimal threshold factor or 
at least an estimate of the number of such configurations is missing and we believe 
it to be a rather challenging problem. Moreover, the techniques introduced in this work can be 
adapted to solve the following integer quadrature problem: let $\mu_{R}$ be a discrete finite 
positive measure supported on $\mathbb{N}$ and which depends on a parameter $R \in [0,\infty[$ 
\begin{equation*}
\mu_{R} = \sum_{j=0}^{\infty} a_j(R) \delta_{j}.
\end{equation*}   
Let us further assume that the zeros of the orthogonal polynomials associated with $\mu_{R}$ are strictly 
increasing functions of the parameter $R$. Then, one can adapt the techniques introduced in this work to compute 
the quantity $R_{m,n}$ defined as the supremum of positive real numbers $R$ such that $\mu_{R}$ admits a positive quadrature 
with integer nodes less or equal to $m$ and which is exact for polynomials of degree at most $n$. Details of this proposed 
solution will appear elsewhere.


\end{document}